\documentclass[11pt,twoside,reqno,centertags,english]{amsart}
\usepackage[oldstylenums]{kpfonts}
\usepackage{lettrine}
\usepackage[utf8]{inputenc}
\usepackage{babel}
\usepackage{bbding}
\usepackage{hyperref}
\usepackage{geometry}
\hypersetup{
colorlinks,
    citecolor=blue,
    filecolor=blue,
    linkcolor=blue,
    urlcolor=blue
    }
\def\be{\begin{equation}}
\def\ee{\end{equation}}
\newcommand {\p} {\partial}
\usepackage{amsmath, amsfonts, amssymb}
\usepackage{amsthm}
\newtheorem{theorem}{Theorem}[section]
\newtheorem{lemma}[theorem]{Lemma}

\newtheorem{corollary}[theorem]{Corollary}
\newtheorem{proposition}[theorem]{Proposition}
\theoremstyle{remark}
\newtheorem{remark}[theorem]{Remark}
\newtheorem{exmp}[theorem]{Example}
\theoremstyle{definition}
\newtheorem{definition}[theorem]{Definition}

\usepackage{mathtools}

\makeatletter
\DeclareRobustCommand\widecheck[1]{{\mathpalette\@widecheck{#1}}}
\def\@widecheck#1#2{%
    \setbox\z@\hbox{\m@th$#1#2$}%
    \setbox\tw@\hbox{\m@th$#1%
       \widehat{%
          \vrule\@width\z@\@height\ht\z@
          \vrule\@height\z@\@width\wd\z@}$}%
    \dp\tw@-\ht\z@
    \@tempdima\ht\z@ \advance\@tempdima2\ht\tw@ \divide\@tempdima\thr@@
    \setbox\tw@\hbox{%
       \raise\@tempdima\hbox{\scalebox{1}[-1]{\lower\@tempdima\box
\tw@}}}%
    {\ooalign{\box\tw@ \cr \box\z@}}}
\makeatother

\DeclareMathOperator{\tr}{Tr}

\DeclareMathOperator{\di}{div}
\newcommand{\bigk}{\mathcal K(E,B)}
\newcommand{\bigkp}{\mathcal K(E_+,B_+)}
\newcommand{\bigkm}{\mathcal K(E_m,B_m)}

\newcommand {\R} {\mathbb{R}}

\numberwithin{equation}{section}

\date{}

\title[Quadratic nonlinearity]{Partial differential equations with quadratic nonlinearities viewed as matrix-valued optimal ballistic transport problems}
\author[D.~Vorotnikov]{Dmitry Vorotnikov}
\address[D.~Vorotnikov]{University of Coimbra, CMUC, Department of Mathematics,  3001-501 Coimbra, Portugal}{}
\email{mitvorot@mat.uc.pt}

\begin{document}

\begin{abstract}  We study a rather general class of optimal ``ballistic'' transport problems for matrix-valued measures. These problems naturally arise, in the spirit of \emph{Y. Brenier. Comm. Math. Phys. (2018) 364(2) 579-605}, from a certain dual formulation of nonlinear evolutionary equations with a particular quadratic structure reminiscent both of the incompressible Euler equation and of the quadratic Hamilton-Jacobi equation. The examples include the ideal incompressible MHD, the template matching equation, the multidimensional Camassa-Holm (also known as the $H(\di)$ geodesic equation), EPDiff, Euler-$\alpha$, KdV and Zakharov-Kuznetsov equations, the equations of motion for the incompressible iso\-tropic elastic fluid and for the damping-free Maxwell's fluid. We prove the existence of the solutions to the optimal ``ballistic'' transport problems. 
For formally conservative problems, such as the above mentioned examples, a solution to the dual problem determines a ``time-noisy'' version of the solution to the original problem, and the latter one may be retrieved by time-averaging.  This yields the existence of a new type of absolutely continuous in time generalized solutions to the initial-value problems for the above mentioned PDE. We also establish a sharp upper bound on the optimal value of the dual problem, and explore the weak-strong uniqueness issue.
\end{abstract}
\maketitle


\textbf{MSC [2020] 35D99, 37K58, 47A56, 49Q22, 76M30}


\section{Introduction}  
It is well-known that Cauchy problems for nonlinear conservative hyperbolic equations, including the ones with quadratic nonlinearities, can have infinitely many weak solutions, cf. \cite{DP79}. This phenomenon has recently attracted lots of attention in the context of convex integration, see, e.g., \cite{lel09,lel10,BNF16,DS17}. No general selection principle for weak solutions has been found so far, although there is some consensus that for physically relevant solutions the energy (that should be formally conserved along the flow, but often fails to do it for the weak solutions) cannot exceed its initial value, cf. \cite{DS17}.

In this connection, for the incompressible Euler and the inviscid Burgers equations, Brenier \cite{CMP18} has recently suggested   to search for the solutions that minimize the time average of the (kinetic) energy. This problem might not always admit a solution, but it leads to a dual variational problem. The dual problem has an optimizer both for the incompressible Euler and the inviscid Burgers. Moreover, he has found explicit formulas that relate the smooth solutions of the incompressible Euler and the inviscid Burgers existing on small time intervals with the solutions of the corresponding dual problems. 

The goal of this paper is to develop Brenier's approach by finding structures in nonlinear quadratic PDE that permit to define similar dual problems and to prove their solvability. Moreover, given a solution of the dual problem, we would like  to be able to construct from it a suitable generalized solution to the original problem that would coincide with the strong solution of the PDE in question, when the latter exists, at least on small time intervals. We will adopt a rather abstract functional analytic setting, and the main assumptions that allow us to achieve these goals are the \emph{formal conservation of energy}, and a certain \emph{trace condition}. In this framework, we can show existence of a solution to the dual problem belonging to a rather natural regularity class, and a suitably defined time average of the latter provides a solution to the original problem with reasonable consistency and regularity properties. More exactly, {\bf the dual problem can be solved in a relaxed sense assuming neither the trace condition nor the formal conservativity, but the trace condition dramatically improves the regularity, whereas the formal conservativity is needed to secure the consistency and the weak-strong uniqueness}. Our abstract approach makes it possible to include a lot of PDE in the theory.  

The proposed abstract problem consists in finding a function $$v:[0,T]\to X^n,$$ in a suitable regularity class, that solves \be \label{e:aeuler}\p_t v= PL(v\otimes v), \quad v(t,\cdot)\in P(X^n), \quad v(0,\cdot)=v_0\in P(X^n).\ee Here $(\Omega,\mathcal A,\mu)$ is a separable probability measure space, $X:=L^2(\Omega)$, $n\in \mathbb N$, $$P:X^n\to X^n$$ is any orthogonal projector (i.e., a self-adjoint idempotent linear operator), and $$L: D(L)\subset X_s^{n\times n}\to X^n$$ is a closed densely defined linear operator. This setting can be further generalized: namely, it is possible to add a linear term to \eqref{e:aeuler},  see Remark \ref{r:aoper}.  

We will mainly focus on the situation when $L$ satisfies the formal \emph{conservativity} condition \be \label{e:acons} (L(v\otimes v), v)=0, \quad v\in P(X^n),\ee provided $v$ is a \emph{sufficiently smooth} vector field (see conventions' section at the end of the Introduction for the meaning of this expression). However, this assumption will not be required everywhere, and we will explicitly specify every instance when it is necessary. 

When \eqref{e:acons} is assumed, we call \eqref{e:aeuler} the \emph{abstract Euler equation}. As we will see, aside from the incompressible Euler and the inviscid Burgers equations studied in \cite{CMP18}, this setting includes the ideal incompressible MHD equations, the template matching equation, the multidimensional Camassa-Holm (also known as the $H(\di)$ geodesic equation), EPDiff, Euler-$\alpha$, KdV and Zakharov-Kuznetsov equations, and two models of motion of incompressible elastic fluids.

Our primal problem is to search for a weak solution to \eqref{e:aeuler} that minimizes the time integral of the energy \begin{equation} \label{e:consp1} K(t):=\frac 12(v(t),v(t)).\end{equation}
We will observe that this problem has a dual formulation that, up to regularity issues, reads \be\label{e:concint}\int_0^T (v_0,q)\, dt-\frac 1 2 \int_0^T (G^{-1}q,q)\, dt \to \sup\ee subject to the constraints \be\label{e:constrint}\p_t G+2(L^*\circ P) q=0, \quad G(T)= I.\ee (the parentheses in \eqref{e:concint} stand for the scalar product in $X^n$). The dual variables are a positive-semidefinite-valued matrix field $G: [0,T]\to X^{n\times n}_s$ and a vector field $q: [0,T]\to X^n$.  

The dual problem \eqref{e:concint}, \eqref{e:constrint} may be viewed as a rather general optimal ballistic transport problem. This problem can be solved for continuous (and ``vanishing at infinity'', in the case of non-compact $\Omega$) initial data under a very mild assumption that the image of the constant function $I$ (the unit matrix of size $n$) belongs to the kernel of the operator $PL$ (\textbf{Theorem \ref{t:exweak}}). This particularly holds if $L$ is a differential operator with vanishing zero-order terms. Our key finding is a certain trace condition (\textbf{Definition \ref{deftr}}) that yields an extra a priori bound, which allows us to solve the dual problem in a natural regularity class (\textbf{Theorem \ref{t:ex}}) and for merely $L^2$ initial data.  Another important finding is that  the conservativity hypothesis \eqref{e:acons} implies a neat link between the original and the dual problems (\textbf{Theorem \ref{t:smooth}}). Namely, a strong solution $v$ to the original problem, if it exists, generates the ``time-noisy'' function $V:=v+(t-T)\partial_t v$. The same $V$ can be obtained by solving the dual problem (which is much easier). Moreover, we can show that there is no duality gap provided there exists a strong solution on $(0,T)$, at least when $T$ is not too large.  This suggests to construct a generalized solution $v$ to the original problem by time-averaging the corresponding function $V$ extracted from the solution of the dual problem. This \textbf{generalized solution automatically belongs to the same regularity class as the strong solutions} (Remark \ref{r:time}). In this context it is possible to establish a sharp upper bound on the optimal value of the dual problem 
(\textbf{Theorem \ref{t:ldistr}}), and to investigate the weak-strong uniqueness and the discrepancy between the original and the dual problems (\textbf{Corollaries \ref{c:wsu} and \ref{c:enlevels}}). The latter results (i.e., the ones of Section \ref{s:uniq}) are \textbf{new even for the incompressible Euler}.

 There exists a celebrated and rather general approach to conservative PDE with quadratic nonlinearities \cite{Ar98,KW08}:  the key idea that goes back to Arnold \cite{A66} is that some of them can be regarded as geodesic equations on infinite-dimensional Lie groups. The latter are known as the Euler-Arnold equations or the Euler-Poincar\'e equations. This approach can be employed to prove local in time existence of solutions, as was done in the famous paper \cite{EM70}. The underlying infinite-dimensional Riemannian structures often have non-negative curvature, which might prevent global existence of smooth solutions, cf. \cite{ra18}. However, in applications it is important to have some global in time solutions, at least in a generalized sense.

The relation between the problem \eqref{e:aeuler} and the dual optimal ballistic transport problem leads us to another unified approach to PDE with quadratic nonlinearities. This approach is less geometric but can be applicable to a larger class of PDE.  Indeed, the problem \eqref{e:aeuler} is more general than the Euler-Arnold equation since it merely requires a quadratic structure of the nonlinearity but not necessarily a hidden rich infinite-dimensional geometry. Moreover, our approach is mainly concerned with generalized solutions on time intervals of arbitrary length and without restrictions on the initial data. 

Anyway, the majority of our examples may be viewed as Euler-Arnold equations, and
the incompressible Euler equation is our basic prototype. 

Several types of global generalized solutions were suggested in the literature for the incompressible Euler equation for $d>2$ and without restrictions on the initial data. These include the weak (distributional) solutions. Actually, only existence of the so-called \emph{wild} distributional solutions that have an instantaneous jump of the kinetic energy has been proved so far \cite{W11}. For convenience, below we refer to this kind of solutions as type A solutions. 

The second type (below referred to as type B) of known solutions to the incompressible Euler equation are the dissipative solutions of Lions \cite{Li96}. These solutions are merely weakly continuous in time and may be obtained as the weak vanishing viscosity limit of the corresponding solutions to the Navier-Stokes equations. 
Another type of solutions to the incompressible Euler equation is the generalized Young measure solution \cite{DiPerna-Majda, Bre11, Sze} (type C). A related but more optimal-transport-inspired approach is the generalized flow of Brenier \cite{B89,Br20} (type D). Both type C and type D are so weak that a value of a solution at any spatio-temporal point is substituted with a probability measure in a relevant way.

Finally, the fifth type (type E) of solutions to the incompressible Euler and inviscid Burgers equations has been recently suggested by Brenier in \cite{CMP18}. In this paper we develop his approach, and construct type E solutions (for arbitrary initial data and in any space dimension $d$) to the ideal incompressible MHD, EPDiff, Euler-$\alpha$, KdV and Zakharov-Kuznetsov equations (only for $d=1,2,3$), to the multidimensional Camassa-Holm (also known as the $H(\di)$ geodesic equation), to the equations of motion of the ``neo-Hookean'' incompressible iso\-tropic elastic fluid and of the damping-free Maxwell's fluid, and (in a weaker sense) to the template matching equation. The consistency and weak-strong uniqueness results announced above, as well as the sharp optimal value bound, are applicable to these PDE. 

Aside from the KdV and Zakharov-Kuznetsov equations, proving the global solvability of the equations listed above faces the same or even stronger difficulties as in the case of the 3D incompressible Euler equation (for $d>2$ and often even for $d=2$). 
Naturally, significant restrictions on the initial data or on the size of the domain can help to prove the existence of global solutions of the above models. For example, one can find such results in \cite{BBS88,caillin,Xu20} and \cite{ST05,ST07,LSZ15,lei2015,lei2016} for the ideal incompressible MHD and the ``neo-Hookean'' incompressible isotropic elastic fluid, respectively. Apart from  that, adding dissipative terms, i.e., making the equations at least partially parabolic, can also secure global well-posedness, as was done in \cite{ER15} for the Maxwell's fluid. 

We are not aware of any theorems about existence of type A or C solutions emanating from arbitrary initial data for the above-mentioned examples. We refer to \cite{BNF16,F18,F21} for some related results concerning existence and non-existence of type A solutions for the ideal incompressible MHD. Some particular weak solutions to the EPDiff equations (with initial data being some specific measures) were discussed in \cite{EPD05}.

 The ideal incompressible MHD and Euler-$\alpha$ are known to have type B solutions  \cite{Wu00,V12}. The quadratic structure of the abstract Euler equation \eqref{e:aeuler} together with the conservativity \eqref{e:acons} comply nicely with  Lions' concept (see \cite[Appendix]{V12} for a related discussion). We have little doubt that all the examples of Section \ref{Sec3} admit type B solutions (this should not be difficult to prove but lies beyond the scope of this article). It would be interesting to find a link between our type E solutions and the type B solutions. However, our type E solutions are obtained in a completely different manner, and are more regular  in time and in space. 

Type D solutions to the the multidimensional Camassa-Holm were recently constructed in \cite{GNV18}. 

\subsection*{Plan of the paper} 
Section \ref{s:opt} clarifies the link of the proposed approach with the theory of optimal transport. In Section \ref{seca}, we describe our basic setting in detail, discuss the dual problem and prove the absence of duality gap on small time intervals in the case where there exists a strong solution.  In Section \ref{s:exi}, we prove the existence theorems and, in the formally conservative case, define the generalized solutions to the abstract Euler equation. The relevance of this definition is supported by Section \ref{s:uniq}, where we establish a sharp upper bound on the optimal value of the dual problem, and explore the weak-strong uniqueness issue. At the end of Section \ref{s:uniq} we discuss the applicability of our theory to the case when an extra linear term is present in \eqref{e:aeuler}. Section \ref{Sec3} examines the above-mentioned PDE examples. 

\subsection*{Basic notation and conventions} 

Let $(\Omega,\mathcal A,\mu)$ be a separable probability measure space. Denote for brevity $X=L^2(\Omega)$. We use the notations $\R^{n\times n}$ and $\R^{n\times n}_s$ for the spaces of $n\times n$ matrices and symmetric matrices, resp., with the scalar product generated by the Frobenius norm.  Let $X^{n\times n}_s$ be the subspace of $X^{n\times n}$ consisting of symmetric-matrix-valued functions. The parentheses $(\cdot,\cdot)$ will stand for the scalar products in $X^n$ and $X^{n\times n}_s$. For $A,B\in X^{n\times n}_s$, we write $A\geq B$ and $A>B$ when $A-B$ is a nonnegative-definite-matrix-function and is a strictly-positive-definite-matrix-function, resp. The action of a matrix-function $A$ from $X^{n\times n}_s$ on a vector-function $\xi$ from $X^n$ is denoted $A.\xi$ or simply $A\xi$. The symbol $I$ in most cases stands for the identity matrix of a relevant size, but sometimes denotes the identity operator in a Hilbert space. 

Fix $n\in \mathbb N$ and the operators $P$, $L$ as above. Let $L^*:D(L^*)\subset X^n\to X_s^{n\times n}$ be the adjoint of $L$. Fix some linear dense subspace $\mathcal R\subset X$. Assume that $$\mathcal R\subset L^\infty(\Omega)$$ and $$\mathcal R^{n}\subset D(L^*)\subset X^n,\ \mathcal R^{n\times n}_s\subset D(L)\subset X^{n\times n}_s,\  L(\mathcal R^{n\times n}_s)\subset \mathcal R^{n},\ L^*(\mathcal R^{n})\subset \mathcal R^{n\times n}_s,\ P(\mathcal R^{n})\subset \mathcal R^{n}.$$ We will abuse the language and call the elements of $\mathcal R$ \emph{sufficiently smooth} functions. For example, if $\Omega$ is a Riemannian manifold, we can take the set of conventional smooth functions as our $\mathcal R$.

Fix also a linear dense subspace $\widehat{\mathcal R}\subset L^2((-\epsilon,T+\epsilon)\times \Omega)$, with a small $\epsilon>0$. Assume that $$\widehat{\mathcal R}\subset L^\infty((-\epsilon,T+\epsilon)\times\Omega), \quad \partial_t \widehat{\mathcal R}\subset\widehat{\mathcal R},\quad \widehat{\mathcal R}(t)=\mathcal R, \ t\in [0,T],$$ and  $$L(\widehat{\mathcal R}^{n\times n}_s)\subset \widehat{\mathcal R}^{n},\ L^*(\widehat{\mathcal R}^{n})\subset \widehat{\mathcal R}^{n\times n}_s,\ P(\widehat{\mathcal R}^{n})\subset \widehat{\mathcal R}^{n}.$$ 
A time-dependent function $\upsilon:[0,T]\to X$ is called sufficiently smooth if $\upsilon\in\widehat{\mathcal R}\big|_{[0,T]}$. 

Sometimes $\Omega$ will be assumed to be a separable locally compact metric space. In this case, $\mathcal M(\Omega)$ will denote the space of finite signed Radon measures on $\Omega$, and $C_0(\Omega)$ will denote of the completion of the space of compactly supported continuous functions on $\Omega$ w.r.t. the supremum norm. Moreover, it will be assumed that $\mathcal R\subset C_0(\Omega)$ and the inclusion is dense w.r.t. the supremum norm. 

Some extra notation to be used exclusively in Section \ref{Sec3} is postponed until the beginning of that section.

\section{The optimal transport perspective}
\label{s:opt}
The goal of this heuristic section is to help the reader to feel the relevance of the abstract structure \eqref{e:aeuler} and of the dual problem \eqref{e:concint}, \eqref{e:constrint} via exploring the links with the theory of optimal transport. 
\subsection{Similarity between the incompressible Euler and the quadratic Hamilton-Jacobi equations} \label{eulhj}
The Euler equations of motion of a homogeneous incompressible inviscid fluid \cite{Ar98} are \begin{gather} \label{euler1} \p_t v+\di (v\otimes v)+\nabla p=0,\\ 
\label{euler2}  \di v=0,\\ \label{euler4}
v(0)=v_0.\end{gather} The unknowns are $v:[0,T]\times \Omega\to \R^d$ and $p:[0,T]\times \Omega\to \R$. Here, for simplicity,  $\Omega$ is the periodic box $\mathbb T^d$. The Euler equations may be rewritten in the form \be \label{e:euler} \p_t v= \mathcal P\mathcal L(v\otimes v), \quad v(t,\cdot)\in \mathcal P(X^d),\quad v(0,\cdot)=v_0\in \mathcal P(X^d),\ee where $X:=L^2(\Omega)$ and $\mathcal P:X^d\to X^d$ is the Leray-Helmholtz projector \cite{Te79}, whereas $$\mathcal L=-\di,\quad \mathcal L: D(\mathcal L)\subset X_s^{d\times d}\to X^d.$$ The kinetic energy $\frac 12\int_\Omega |v|^2(t,x)\,dx$ is formally conserved due to \be \label{e:cons} (\mathcal L(v\otimes v), v)=0, \quad v\in \mathcal P(X^d)\ee for any  sufficiently smooth vector field $v$. Hence, the incompressible Euler fits into our abstract framework \eqref{e:aeuler} and satisfies \eqref{e:acons}. 

The quadratic Hamilton-Jacobi equation \be \label{e:hj} \p_t \psi+\frac 1 2 |\nabla \psi|^2=0, \ \psi(0)=\psi_0.\ee  plays an important role in the theory of Monge-Kantorovich optimal transport \cite{villani03topics,villani08oldnew,S15}. The unknown is $\psi:[0,T]\times \Omega\to \R$ with $\Omega$ as above.  For our purposes, it is instructive to rewrite \eqref{e:hj} in terms of the artificial matricial variable $$U(t)=-\frac 2 d \psi(t) I+ \int_0^t \left[(\nabla \psi) \otimes (\nabla \psi)- \frac 1 d |\nabla \psi|^2 I\right];$$ the result is \be \label{e:hju}  \p_t U=\frac 1 4 ((\nabla \tr U) \otimes (\nabla \tr U)), \ U(0)=-\frac 2 d \psi_0 I.\ee  Now, setting $v=\nabla \psi= -\frac 1 2 \nabla \tr U$, we recast \eqref{e:hj} as \be \label{e:hjv} \p_t v+\frac 1 2 \nabla \tr (v\otimes v)=0, \ v(0)=\nabla \psi_0.\ee For $d=1$, \eqref{e:hjv} coincides with the inviscid Burgers equation. Setting $P=I$, $L=-\frac 1 2 \nabla \tr$, we recover our abstract equation \eqref{e:aeuler}. Note that \eqref{e:acons} does not hold and the energy $\frac 1 2(v,v)$ is not conserved except for $d=1$. 

We have just reformulated the Hamilton-Jacobi equation so that it became reminiscent of  the incompressible Euler equation. The converse of this observation is the following one. Given an initial datum $v_0$ for the incompressible Euler equation with zero mean, one can find $V_0\in X^{d\times d}_s$ so that $v_0=-\mathcal P \di V_0$ (cf. Remark \ref{revers}). Then the solution to \eqref{e:euler} can at least formally be expressed in terms of a new symmetric-matrix-valued variable $V$ such that $v=-\mathcal P \di V$: \be \label{e:eulerne} -\mathcal P\di \p_t V=- \mathcal P \di (( \mathcal P  \di V) \otimes (\mathcal P  \di V)).\ee Without loss of generality (cf. Remark \ref{revers}), $V$ may be selected in such a way that \be \label{e:eulernew}  \p_t V=(( \mathcal P  \di V) \otimes (\mathcal P  \di V)), \ V(0)=V_0.\ee This embodiment of the incompressible Euler equation resembles the Hamilton-Jacobi equation in its matricial formulation \eqref{e:hju}.


\subsection{The dual problem and the optimal ballistic transport} As anticipated in the Introduction, our primal problem consists in finding a weak solution to \eqref{e:aeuler} that minimizes the time integral of the energy, and the dual problem formally reads as \eqref{e:concint}, \eqref{e:constrint}. 

Assume for a while that the initial data satisfy the following extra assumption \be \label{e:goodv} v_0\in PL( X_s^{n\times n}).\ee  Then we have \be \label{e:inpsi} v_0=PL U_0 \ee for some $U_0\in X^{n\times n}_s$ (cf. Remark \ref{revers}). Consequently, for a maximizer $(G,q)$ to \eqref{e:concint}, \eqref{e:constrint} we deduce $$\int_0^T (v_0,q)\, dt=-\frac 1 2\int_0^T (U_0,\p_t G) \, dt=\frac 1 2 (U_0,G(0))- \frac 1 2 (U_0,I).$$ Since the last term does not depend on $G,q$, \eqref{e:concint} becomes\be\label{e:concintb}\frac 1 2 (U_0,G(0))-\frac 1 2 \int_0^T (G^{-1}q,q)\, dt \to \sup.\ee

For the Hamilton-Jacobi equation \eqref{e:hjv} the constraints read \be\label{e:constrinth1}\p_t G+(\di q) I=0, \quad G(T)= I.\ee Hence, there exists a non-negative scalar function $\rho$ such that $G=\rho I$.  Moreover, \eqref{e:inpsi} obviously holds with $U_0=-\frac 2 d \psi_0 I$.  The dual problem \eqref{e:concintb}, \eqref{e:constrint} for the Hamilton-Jacobi equation  may be rewritten as \be\label{e:concinth}   -\int_{\Omega} \psi_0 \rho(0)\,dx  -\frac 1 2 \int_0^T\int_{\Omega} \rho^{-1}|q|^2\,dx \, dt \to \sup \ee subject to the constraints \be\label{e:constrinth}\p_t \rho+\di q=0, \quad \rho(T)=1,\quad \rho\geq 0.\ee This is quite similar to the dynamical formulation of the Monge-Kantorovich optimal transport \cite{BB00,villani03topics,S15}, where $\rho(t)$ describes the evolution of the transported density; the only difference is that instead of prescribing the initial density $\rho(0)$ we have an extra term in \eqref{e:concinth} involving $\rho(0)$ and $\psi_0$. More accurately, \eqref{e:concinth}, \eqref{e:constrinth} can be viewed as an optimal \emph{ballistic} transport problem for a probability density $\rho(t)$. We have borrowed this phrasing from \cite{BG19}, where related problems were investigated. In the same spirit, \eqref{e:concint}, \eqref{e:constrint} can be regarded as an optimal ballistic transport problem for a positive-definite-valued matrix-valued density $G(t)$.

\begin{remark}[Ballistic vs. non-ballistic transport] The non-ballistic counterpart of our problem \be\label{e:concintnb}-\frac 1 2 \int_0^T (G^{-1}q,q)\, dt \to \sup\ee under the constraints \be\label{e:constrintnb}\p_t G+2(L^*\circ P) q=0, \quad G(0)=G_0,\ G(T)= G_T,\ee
\be\label{e:bweint} G\ \mathrm{is}\ \mathrm{a}\  \mathrm{positive-semidefinite}\ \mathrm{matrix-valued}\ \mathrm{density},\ee can fail to be solvable in any weak sense at this level of generality (even for $P=I$), because the supremum in \eqref{e:concintnb} can be $-\infty$, cf. \cite[Introduction]{BV18}, and it is plausible that it can also be $0$ (which is not compatible with the existence of optimizers), in line with \cite{MM5, MM06, BHM12, SV19}. The ballistic situation is rather different: we a priori know that the supremum in \eqref{e:concint} is non-negative, and it is important to be able to avoid that it becomes $+\infty$, cf. Remark \ref{plusin}. \end{remark}

\begin{remark}
The trace condition (see Definition \ref{deftr}) and the conservativity hypothesis \eqref{e:acons}, which are so important for our paper, have nothing to do with the classical optimal transport, since both of them fail for the problem \eqref{e:concinth}, \eqref{e:constrinth}, and become relevant only in our general framework.  \end{remark}

\subsection{The bidual problem} 
It is well-known, cf. \cite{BB07, S15}, that the dynamical Monge-Kanto\-rovich optimal transport admits a dual formulation. The corresponding dual formulation of the ballistic problem \eqref{e:concinth}-\eqref{e:constrinth} formally reads \be\label{e:bidualh}   -\int_{\Omega} \varphi(T)\,dx  \to \inf \ee subject to the constraints \be\label{e:bdch}\p_t \varphi+\frac 1 2 |\nabla \varphi|^2\leq 0, \quad \varphi(0)\leq \psi_0 \ \mathrm{a.e.}\ \mathrm{in}\ \Omega.\ee Note that this may be regarded both as a convex relaxation of \eqref{e:hj} and as a bidual problem to the Hamilton-Jacobi equation in its ``Eulerian'' layout \eqref{e:hjv}. It is also worth to consider the analogous relaxation of \eqref{e:hju}: \be\label{e:bidualhu}   \int_{\Omega} \tr\,V(T)\,dx  \to \inf \ee subject to the constraints \be\label{e:bdchu} \frac 1 4 ((\nabla \tr V) \otimes (\nabla \tr V)) \leq \p_t V, \ V(0)\leq -\frac 2 d \psi_0 I \ \mathrm{a.e.}\ \mathrm{in}\ \Omega.\ee More generally, taking for granted \eqref{e:inpsi}, by mimicking the corresponding derivation of the dual dynamical Monge-Kantorovich optimal transport, see e.g. \cite{S15}, it is possible to formally derive the following abstract problem: \be\label{e:bidual}   \int_{\Omega} \tr\,\Xi(T)\,dx  \to \inf \ee subject to the constraints \be\label{e:bdc}  (PL \Xi) \otimes (PL \Xi)\leq  \p_t \Xi, \ \Xi(0)\leq U_0 \ \mathrm{a.e.}\ \mathrm{in}\ \Omega.\ee This problem is dual to \eqref{e:concintb}, \eqref{e:constrint} and thus bidual to \eqref{e:aeuler}. In particular, for the incompressible Euler equation we obtain a convex relaxation of \eqref{e:eulernew}.

\begin{remark}[The bidual problem as a relaxation of the original one] \label{revers} Let us show that the bidual problem may be regarded as a convex relaxation of \eqref{e:aeuler}. Assume \eqref{e:goodv}. Then there is $U_0\in X^{n\times n}_s$ verifying \eqref{e:inpsi}.  By the way, when $PL$ has closed range in $P(X^n)$, \eqref{e:goodv} holds if and only if $v_0$ is orthogonal to the kernel of $L^*$ in $P(X^n)$. We claim that \eqref{e:aeuler} is at least formally equivalent to \be\label{e:aemat}\p_t U=  (PL U) \otimes (PL U), \ U(0)=U_0.\ee  Indeed, any solution $U$ to \eqref{e:aemat} obviously generates a solution $v=PL U$ to \eqref{e:aeuler}. Conversely, \eqref{e:inpsi} and \eqref{e:aeuler} imply that $v(t)=PL V(t)$ for some $V(t)\in X^{n\times n}_s$. Moreover, \eqref{e:aeuler} implies \be\label{e:aemat1}\p_t V=  (PL V) \otimes (PL V)+\xi\ee with some $\xi(t)\in X^n$, $PL\xi=0$.  Now $$U(t)=V(t)- \int_0^t \xi$$ solves \eqref{e:aemat}. It is apparent that the bidual problem \eqref{e:bidual}, \eqref{e:bdc}  may be viewed as a convex relaxation of \eqref{e:aemat}. This argument is applicable to the incompressible Euler equation on the torus if $v_0$ has zero mean (such $v_0$ are orthogonal to the kernel of $\mathcal L^*$, and $\mathcal P\mathcal L$ has closed range in $\mathcal P(X^d)$).  \end{remark}

\begin{remark}[Convex relaxation by dint of subsolutions] \label{convint}Another natural convex relaxation of the primal problem is to search for the \emph{subsolution} $(v,M)$ to \eqref{e:aeuler} that minimizes the ``energy'' $\frac 1 2 \tr M$, cf. \cite{CMP18,Br20}. Here the subsolutions\footnote{This definition of subsolutions does not take into account the corresponding \emph{wave cone} that plays a significant role in the theory of convex integration, cf. \cite{lel09,lel10}. However, in the case of the incompressible Euler equation this definition \textbf{is equivalent} \cite{DS17} to the conventional definition of subsolutions. When the wave cone is not very large, cf. \cite{F18,Sze12}, there is a discrepancy between the two definitions.} are the pairs $(v,M)$, $v\otimes v \leq M$, that satisfy \be \label{e:aeulersub}\p_t v= PL(M), \quad v(t,\cdot)\in P(X^n), \quad v(0,\cdot)=v_0\in P(X^n)\ee in a weak sense. The resulting problem may be viewed as dual to \eqref{e:concint}, \eqref{e:constrint} in the spirit of Fenchel--Rockafellar. Moreover, in the case of the incompressible Euler, its optimal value coincides \cite[Theorem 2.6]{CMP18} with the optimal value of \eqref{e:concint}, \eqref{e:constrint}. However, Brenier's proof does not seem to be easily extendable to the general situation and even to the inviscid Burgers. \end{remark}

The study of the bidual problem and/or of the problem from Remark \ref{convint} lies out of the scope of this article: we present them here just to emphasize the link between the original and the dual problems.  

\section{The dual problem and its consistency} \label{seca} 

In this section, we describe our basic setting. The necessity of the formal conservativity assumption \eqref{e:acons} is explicitly emphasized where needed. We define the strong and weak solutions to the original problem \eqref{e:aeuler}. We derive the optimal ballistic transport problem, also referred to as the dual problem, that is pivotal for the whole paper. After that, we prove that every strong solution corresponds to a solution of the dual problem, at least if $T$ is sufficiently small. We will further explore the consistency issue in Section \ref{s:uniq}. We finish Section \ref{seca} by some remarks and a simple example. 

The quadratic problem \eqref{e:aeuler} admits the following natural weak formulation: \be \label{e:w1}\int_0^T \left[(v,w)+(v,\p_t a)+(v\otimes v, L^*a)\right]\, dt+(v_0, a(0))=0\ee for all sufficiently smooth vector fields $a: [0,T]\to P(X^n)$, $a(T)= 0$, $w: [0,T]\to (I-P)(X^n)$.

We now observe that \eqref{e:acons} implies $$ ((u+r)\otimes (u+r), L^*(u+r))-
((u-r)\otimes (u-r), L^*(u-r))-2(r\otimes r, L^*r)=0$$ for  $u,r\in P(X^n)$ sufficiently smooth, whence \be \label{polar}(u\otimes u, L^* r)+2\left(r\otimes u, L^* u\right)=0.\ee
Consequently, \eqref{e:w1} can be formally recast as \be \label{e:w1n}\int_0^T \left[(v,w)+(v,\p_t a)-2(a\otimes v, L^*v)\right]\, dt+(v_0, a(0))=0.\ee This implies the ``strong'' reformulation of the abstract Euler equation \eqref{e:aeuler} that is valid only in the conservative case:  
\be \label{e:aeuler2}\p_t v+2P[L^*v.v]=0, \quad v(t,\cdot)\in P(X^n), \quad v(0,\cdot)=v_0.\ee

Formally, \eqref{e:acons} implies that the energy \eqref{e:consp1} is conserved, which yields \begin{equation} \label{e:consp2}  \frac 1 T \int_0^T K(t)=K_0:=\frac 12(v_0,v_0).\end{equation}

Let us now rewrite problem \eqref{e:w1} in terms of the test functions $B:=L^*a$ and $E:=\p_t a+w$ (note that the conservativity \eqref{e:acons} is not needed at this stage). We first observe that \be \label{e:vab} (v_0, a(0))=-\int_0^T \left(v_0,   \p_t a\right)=-\int_0^T \left(v_0,   E\right) \ee since $w$ is orthogonal to $v_0$. The link between $B$ and $E$ can alternatively be described by the conditions \be\label{e:constr1}\p_t B=(L^*\circ P) E, \quad B(T)=0.\ee Indeed, any pair $(E,B)$ satisfying \eqref{e:constr1} generates a pair $(a,w)$ such that $B=L^*a,\, E=\p_t a+w$, and vice versa. It suffices to take $a(t)=\int_T^t PE$, $w=E-PE$.  Hence, \eqref{e:w1} becomes \be \label{e:w2}\int_0^T \left[(v-v_0,E)+(v\otimes v, B)\right]\, dt=0\ee for all sufficiently smooth vector fields $B: [0,T]\to X^{n\times n}_s$, $E: [0,T]\to X^n$ satisfying the constraints \eqref{e:constr1}. 

For a technical reason, we now need to extend the class of test functions in \eqref{e:w2}. Observe that \eqref{e:constr1} can be rewritten in the following weak form \be \label{e:constrweak}\int_0^T \left[(B,\p_t \Psi)+(E, PL \Psi)\right]\, dt=0\ee for all sufficiently smooth vector fields $\Psi: [0,T]\to X^{n\times n}_s$, $\Psi(0)= 0$. 

Motivated by the discussion above, we adopt the following definitions, where we tacitly assume  $v_0\in P(X^n)$. 
\begin{definition}[Weak solutions] A function $v\in L^2((0,T)\times \Omega;\R^n)$ is a \emph{weak solution} to \eqref{e:aeuler} if it satisfies \eqref{e:w2} for all pairs \be \label{e:be} (E,B)\in L^2((0,T)\times \Omega;\R^n)\times L^\infty((0,T)\times \Omega;\R^{n\times n}_s)\ee meeting the constraint \eqref{e:constrweak}. \end{definition}
\begin{definition}[Strong solutions] \label{d:strongclass} Assume \eqref{e:acons}. A function $v\in L^2((0,T)\times \Omega; \R^n)$ satisfying \begin{equation}\label{e:strongclass} (T-t)v\in AC^2([0,T];X^n), \quad (T-t)L^* v \in L^\infty((0,T)\times \Omega; \R^{n\times n}_s) \end{equation} is a \emph{strong solution} of the abstract Euler equation \eqref{e:aeuler} if \eqref{e:aeuler2} holds. \end{definition}
The relevance of the choice of the regularity class in Definition \ref{d:strongclass} will become apparent below, see Theorem \ref{t:smooth} as well as Sections \ref{s:exi} and \ref{s:uniq}.

We also require  the following technical definition. 
\begin{definition}[Approximation condition] The operator $L$ is said to satisfy  the approximation condition if for every $v\in P(X^n)$ with $L^* v\in L^\infty(\Omega, \R^{n\times n}_s)$ there is a sequence $v_m\to v$ in $X^n$ such that $v_m\in \mathcal R^n\cap P(X^n)$, and \begin{equation*} \|L^* v_m\|_{L^\infty(\Omega, \R^{n\times n}_s)}\leq  C,\end{equation*} where the constant $C$ does not depend on $m$. \end{definition}
Note that this is a natural and not that restrictive assumption: for instance, if $L^*$ is a differential operator, and if either $P(X^n)$ is the kernel of another differential operator or simply $P(X^n)=X^n$, then a standard mollifying procedure provides a required sequence. 

\begin{proposition}[First properties of strong solutions] \label{propap} If \eqref{e:acons} and the approximation condition are assumed, then for any strong solution $v$ the energy is conserved, i.e.,  $K(t)=cst=K_0$ for all $t\in [0,T)$. Moreover, in this case $v$ is also a weak solution.  \end{proposition}
\begin{proof} For any $t\in (0,T)$ up to a null set, $v(t)\in P(X^n)$ and $L^* v(t)\in L^\infty(\Omega, \R^{n\times n}_s)$. Due to the approximation condition, there is a suitable approximating sequence $v_m\in \mathcal R^n\cap P(X^n)$, $v_m\to v(t)$ in $X^n$, with $\|L^* v_m\|_{L^\infty(\Omega, \R^{n\times n}_s)}\leq  C$.  Assumption \eqref{e:acons} implies that  $(v_m\otimes v_m, L^*v_m)=0$. Up to a subsequence, $L^* v_m\to L^* v(t)$ weakly-$*$ in $L^\infty(\Omega, \R^{n\times n}_s)$ (we can identify the limit because $L^*$ is closed and thus weakly closed). Hence, \be \label{e:w3++a}(v\otimes v, L^*v)=0\ee  a.e. in $(0,T)$.   Taking the $X^n$ scalar product of \eqref{e:aeuler2} with $v$ for a.a. $t\in (0,T)$ and employing \eqref{e:w3++a} we get \begin{equation*}  \frac d {dt} K(t)=0\ \textrm{a.e.}\ \textrm{in}\ (0,T).\end{equation*} 

 Since $v=Pv$,  it suffices to establish \eqref{e:w2}
for the pairs $(E,B)$ from the class \eqref{e:be} that satisfy $E=PE$ and \eqref{e:constrweak}. Set $a=\int_T^t E\in AC^2([0,T];X^n)$. Integrating by parts in \eqref{e:constrweak}, we derive that $L^*a=B\in L^\infty((0,T)\times\Omega, \R^{n\times n}_s)$. This implies that \eqref{e:w3++a} also holds if $v$ is replaced by $a$, $v+a$ or $v-a$. Mimicking the proof of \eqref{polar} above, we get $$(v\otimes v, L^* a)+2\left(a\otimes v, L^* v\right)=0$$ for a.a. $t\in (0,T)$.  Consequently, \be \label{e:cla1} (\p_t v, a)=(v\otimes v, L^*a)\in L^1(0,1).\ee On the other hand,  \be \label{e:cla2} (v, \p_t a)=(v,E)\in L^1(0,1).\ee We will prove in a moment that we can legitimately integrate by parts: \be\label{e:claima}\int_0^T(\p_t v, a)=-(v_0,a(0))-\int_0^T(v, \p_t a).\ee Consequently, $v$ and $a$ satisfy \eqref{e:w1} (with $w\equiv 0$). Since \eqref{e:vab} is still valid at this level of generality, we readily derive \eqref{e:w2} for the pair $(E, B)$. 

It remains to prove \eqref{e:claima}. Consider the scalar function $F(t):=(v(t),a(t))$. By \eqref{e:cla1} and \eqref{e:cla2}, $F\in W^{1,1}(0,T)$. Without loss of generality, we can assume that $F$ is continuous at $t=T$ (and hence on $[0,T]$). We now claim that  \be\label{e:cla3}\lim_{\epsilon\downarrow 0}\frac 1 \epsilon \int_{T-\epsilon}^T F =0. \ee Indeed, employing the Cauchy-Schwartz and Hardy's inequalities we estimate \begin{multline}\frac 1 \epsilon \int_{T-\epsilon}^T (v,a)\leq \|v\|_{L^2(T-\epsilon,T; X^n)}\sqrt{\int_{T-\epsilon}^T\frac 1 {\epsilon^2} (a,a)\, dt}\\
\leq \|v\|_{L^2(T-\epsilon,T; X^n)}\sqrt{\int_{T-\epsilon}^T\frac 1 {(T-t)^2} (a,a)\, dt} \leq 2 \|v\|_{L^2(T-\epsilon,T; X^n)}\|\p_t a\|_{L^2(T-\epsilon,T; X^n)}\to 0
\end{multline} as $\epsilon\downarrow 0$.
From \eqref{e:cla3} we immediately conclude that $F(T)=0$ and $$\int_0^T F^\prime(t)\, dt=-F(0),$$ which is equivalent to \eqref{e:claima}.
 \end{proof}

The conservativity property \eqref{e:consp2} can fail for the weak solutions of the abstract Euler equation \eqref{e:aeuler}.  The idea of Brenier \cite{CMP18, Br20}, which we reemploy here, is to search for the solution that minimizes $\int_0^T K(t)$. This is a natural step at least for Euler-Arnold equations, such as the incompressible Euler, ideal incompressible MHD, multidimensional Camassa-Holm, EPDiff, Euler-$\alpha$ and KdV equations, because of their geodesic nature. However, the following derivation of the dual problem does not necessarily require \eqref{e:acons} (see Remark \ref{consdis} for a discussion). 
 
The suggested approach of selecting a solution that minimizes $\int_0^T K(t)$ can be implemented via the saddle-point problem \be\label{e:sadd1}\mathcal I(v_0, T)=\inf_{v}\sup_{E,B:\,\eqref{e:constrweak}}\int_0^T \left[(v-v_0,E)+\frac 1 2(v\otimes v, I+2B)\right]\, dt.\ee The infimum in \eqref{e:sadd1} is taken over all $v\in L^2((0,T)\times \Omega;\R^n)$, and the supremum is taken over all pairs $(E,B)$ satisfying \eqref{e:be} and the linear constraint \eqref{e:constrweak}.

The dual problem is \be\label{e:sadd2}\mathcal J(v_0, T)=\sup_{E,B:\,\eqref{e:constrweak}}\inf_{v}\int_0^T \left[(v-v_0,E)+\frac 1 2(v\otimes v, I+2B)\right]\, dt,\ee where $v,E,B$ are varying in the same function spaces as above.  \begin{remark}[Simple observations about $\mathcal I$ and $\mathcal J$] Since $\inf\sup\geq \sup\inf$, one has $\mathcal I(v_0, T)\ge \mathcal J(v_0, T)$. Note that the $\sup$ in \eqref{e:sadd1} is always $+\infty$ if $v$ is not a weak solution, whence $\mathcal I(v_0,T)=+\infty$ if there are no weak solutions. On the other hand, if \eqref{e:acons} holds, and there exists a strong solution $v$, then the corresponding $\sup$ in \eqref{e:sadd1} is equal to $\frac 12\int_0^T (v\otimes v, I)\, dt=TK_0$, which yields $\mathcal I(v_0,T)\leq TK_0$. \label{rweak}\end{remark}

\begin{remark}[Equivalent reformulation of the dual problem]The goal of this important remark is to provide an equivalent reformulation of the saddle-point problem \eqref{e:sadd2} that will be more convenient from the technical perspective. It is easy to see that any solution to \eqref{e:sadd2} necessarily satisfies \be\label{e:bwe} I+2B\geq 0\ \mathrm{a.e.}\ \mathrm{in}\ (0,T)\times \Omega.\ee Consider the nonlinear functional $$\mathcal K: L^2((0,T)\times \Omega;\R^n)\times L^\infty((0,T)\times \Omega;\R^{n\times n}_s)\to \R$$ defined by the formula
 \be\label{e:defk1} \mathcal K(E,B)=\inf_{z\otimes z\le M}\int_0^T \left[(z,E)+\frac 1 2(M, I+2B)\right]\, dt,\ee  where the infimum is taken over all pairs $(z,M)\in L^2((0,T)\times \Omega;\R^n)\times L^1((0,T)\times \Omega;\R^{n\times n}_s)$. Note that $\bigk$ is a negative-semidefinite quadratic form w.r.t. $E$ for fixed $B$, cf. the proof of Theorem \ref{t:exweak}. We claim that \eqref{e:sadd2} is equivalent to \be\label{e:conc}\mathcal J(v_0, T)=\sup_{E,B:\,\eqref{e:constrweak},\eqref{e:bwe}}-\int_0^T (v_0,E)\, dt+\bigk,\ee the supremum is taken over all pairs $(E,B)$ belonging to the class \eqref{e:be}. Indeed, assume for a while that \be\label{e:bst} I+2B> 0\ \mathrm{a.e.}\ \mathrm{in}\ (0,T)\times \Omega\ \mathrm{and}\ (I+2B)^{-1}E\in L^2((0,T)\times \Omega;\R^n).\ee Then \be
 \label{e:defk}\inf_{v}\int_0^T \left[(v,E)+\frac 1 2(v\otimes v, I+2B)\right]\, dt=-\frac 12\int_0^T((I+2B)^{-1}E,E)\,dt =\bigk,\ee and the first infimum is achieved at $v=-(I+2B)^{-1}E$. 
 Hence our claim is true provided \eqref{e:bst} holds. Let us now drop that assumption and examine the general case. Using \eqref{e:sadd2}, \eqref{e:defk1} and \eqref{e:defk}, we conclude that \begin{multline*} \mathcal J(v_0, T)\leq \sup_{E,B:\,\eqref{e:constrweak}}\inf_{\epsilon>0}\,\inf_{v}\int_0^T \left[(v-v_0,E)+\frac 1 2(v\otimes v, I+2(B+\epsilon I))\right]\, dt \\ =\sup_{E,B:\,\eqref{e:constrweak},\eqref{e:bwe}}\left\{- \int_0^T (v_0,E)\, dt+\inf_{z\otimes z\le M,\, \epsilon>0}\int_0^T \left[(z,E)+\frac 1 2(M, I+2(B+\epsilon I))\right]\,dt\right\}\\=\sup_{E,B:\,\eqref{e:constrweak},\eqref{e:bwe}}\left\{- \int_0^T (v_0,E)\, dt+\inf_{z\otimes z\le M}\int_0^T \left[(z,E)+\frac 1 2(M, I+2B)\right]\,dt\right\}\\=\sup_{E,B:\,\eqref{e:constrweak},\eqref{e:bwe}}-\int_0^T (v_0,E)\, dt+\bigk\leq \mathcal J(v_0, T).  \end{multline*}
Note that the last inequality holds because the infimum in \eqref{e:defk1} is taken over a larger set in comparison with \eqref{e:sadd2}.\end{remark}  

\begin{remark}[More conventional variables] In view of \eqref{e:defk}, the dual problem \eqref{e:conc} can be at least formally rewritten in the form \eqref{e:concint}, \eqref{e:constrint} if we set $G=I+2B$, $q=-E$. Hereafter, we keep using the variables $E, B$. \end{remark}

The following theorem shows that a strong solution to the abstract Euler equation on a small time interval $[0,T]$ determines a solution to the optimization problem \eqref{e:conc}, and vice versa. This advocates the possibility to view the maximizers of \eqref{e:conc} as generalized variational solutions to \eqref{e:aeuler}, at least in the conservative case. Remarks \ref{consdis}, \ref{r:time}, \ref{r:brenier} and Section \ref{s:uniq} will continue this discussion. 

\begin{theorem}[Consistency] \label{t:smooth} Assume \eqref{e:acons} and that the approximation condition holds for $L$. Let $v$ be a strong solution to \eqref{e:aeuler}, satisfying \be\label{e:pd++} I\ge 2(t-T)L^*v(t)\ \textrm{a.e.}\ \textrm{in}\ (0,T)\times \Omega.\ee  Then $\mathcal I(v_0,T)=\mathcal J(v_0,T)=TK_0$. The pair $(E_+,B_+)$ defined by $$B_+=L^*a,\, E_+=\p_t a+w,$$ where \be a=(T-t)v,\, w=2(t-T)(I-P)[L^*v.v],\ee belongs to the class \eqref{e:be} and maximizes \eqref{e:conc}. Moreover, one can invert these formulas and express $v$ in terms of $E_+$ according to
 \be \label{e:ta1} v(t)=\frac 1 {T-t} \int_{t}^T (-PE_+)(s)\,ds,\quad t< T.\ee

\end{theorem}

\begin{proof} 
 By construction, the pair $(E_+,B_+)$ belongs to $L^2((0,T)\times \Omega;\R^n)\times L^\infty((0,T)\times \Omega;\R^{n\times n}_s)$ and verifies  \eqref{e:constrweak}. Moreover,  \eqref{e:pd++} implies \eqref{e:bwe} for $B_+$. Let us observe that \be \label{e:vret1} v+2B_+.v+E_+=0.\ee  Indeed, using \eqref{e:aeuler2} we compute \begin{multline}v+2B_+.v+E_+=v+2(T-t)L^*v.v+(-v+(T-t)\p_t v)+2(t-T)(I-P)[L^*v.v]\\ = (T-t)\p_t v+2(T-t)P[L^*v.v]=0.\end{multline}
 
On the other hand, by Proposition \ref{propap} $v$ satisfies \eqref{e:w2} with test functions $(E_+,B_+)$. Thus we have \be \label{e:w2+}\int_0^T \left[(v-v_0,E_+)+(v\otimes v, B_+)\right]\, dt=0.\ee Hence, by \eqref{e:vret1}, \be \int_0^T \left[-(v_0,E_+)+(v\otimes v, B_+)\right]\, dt=\int_0^T (v\otimes v, (I+2B_+))\, dt,\ee whence  \be \label{e:w3+}\int_0^T \left[(v_0,E_+)+(v\otimes v, B_+)\right]\, dt=-\int_0^T (v\otimes v, I)\, dt.\ee

By Remark \ref{rweak}, we have $\mathcal I(v_0, T)\leq TK_0$. Thus, it suffices to show that \be \label{e:claim1} \int_0^T -(v_0,E_+)\, dt+\bigkp=TK_0, \ee so that there is no duality gap. Indeed,  \eqref{e:vret1} and \eqref{e:w3+} yield \begin{multline*} \int_0^T -(v_0,E_+)\, dt+\bigkp\\=\int_0^T -(v_0,E_+)\, dt+\inf_{z\otimes z\le M}\int_0^T \left[-(z,(I+2B_+)v)+\frac 1 2(M, I+2B_+)\right]\, dt\\=\int_0^T -\left[(v_0,E_+)+\frac 12((I+2B_+)v,v)\right]\, dt\\
=\int_0^T \left[(v\otimes v, I)-\frac 12(v,v)\right]\, dt=TK_0 \end{multline*} because the energy conservation holds for the strong solutions.

Finally, \be-PE_+=-\p_t a=v+(t-T)\p_t v,\ee so \be\int_t^T -PE_+(s)\,ds=\int_t^T [v(s)+ (s-T)\p_s v]\, ds=(T-t)v(t),\ee providing \eqref{e:ta1}.

\end{proof}

\begin{remark}[Solution of the abstract Euler problem as the optimal transportation ``velocity''] \label{c:strong} If in Theorem \ref{t:smooth} one has \be\label{e:pd+} I>2(t-T)L^*v(t)\ \textrm{a.e.}\ \textrm{in}\ (0,T)\times \Omega,\ee then the solution can be also retrieved by the formula \be \label{e:vret} v=-(I+2B_+)^{-1}E_+,\ee cf. \cite[Theorem 2.3]{CMP18}. Indeed, it suffices to observe that  \eqref{e:pd+} means that $I+2B_{+}>0$, and if this holds, \eqref{e:vret} is equivalent to \eqref{e:vret1}. If we pass to the variables $G$ and $q$ as in \eqref{e:concint}, then $v=G^{-1}q$, which generalizes the usual relation between the velocity, density and momentum in the classical optimal transport. We will revisit Brenier's formula \eqref{e:vret} in Remark \ref{r:brenier}.  \end{remark}
\begin{remark}[Non-conservative problems] \label{consdis} It is worth considering the dual problem in the non-conservative case, since one still gets interesting concave maximization problems. Remember that in the case of the Hamilton-Jacobi equation discussed in Section \ref{eulhj} we have obtained an optimal ballistic transport problem. In the conservative framework at least for not very wild solutions one expects $\mathcal J(v_0, T)\leq\mathcal I(v_0, T)\lesssim K_0 < +\infty,$ but for some non-conservative problems it happens that $\mathcal J(v_0, T)=+\infty$, see Example \ref{ex:h} below. At least in some non-conservative situations the minimization of $\int_0^T K(t)$ may be heuristically justified and/or the absense of duality gap may be proved, cf. \cite{Br20}, but the neat formula \eqref{e:ta1} fails to hold without \eqref{e:acons}. In this article we are mainly concerned with conservative examples, and the consistency issue for non-conservative problems lies beyond our scope (the only exception to this assertion is Example \ref{ex:h}). \end{remark}

\begin{exmp}[``Ballistic information transport''] \label{ex:h} Let us examine the simple but instructive example \be \label{e:odeq} \p_t v+\frac 1 2 v^2=0, \ v(0)=v_0,\quad v(t), v_0 \in X.\ee We opt for considering it separately from the conservative PDE examples of Section \ref{Sec3} because the calculations are explicit, and it is not conservative. It is clear that setting $L=-\frac 1 2 I$, $P=I$ we immediately get \eqref{e:aeuler}, and \eqref{e:goodv} is always valid.  

Assume first that \be \label{tv2} Tv_0+2>0\ \mathrm{a.e.}\ \mathrm{in}\ \Omega.\ee Then the explicit solution of the IVP  \eqref{e:odeq} is  $$v(t)=\frac {2v_0}{2+t v_0}\in AC^2_{loc} ([0,T),X).$$ Although we cannot rigorously call it a strong solution (because the latter notion was only defined in the conservative case), it belongs to the regularity class of Definition \ref{d:strongclass}. Moreover, \be\label{e:iifr} \mathcal I(v_0, T)= \int_0^T K(t)= \int_0^T \int_{\Omega}\frac 1 2 v^2\, d\mu\, dt=\int_{\Omega} \left[v_0-\frac {2v_0}{2+T v_0}\right] \, d\mu.\ee The dual problem reads \be\label{e:concinfr}   -\int_{\Omega} v_0\rho(0)\,d\mu  -\frac 1 2 \int_0^T\int_{\Omega} \rho^{-1}|q|^2\,d\mu \, dt \to \sup \ee subject to the constraints \be\label{e:constrinfr}\p_t \rho=q, \quad \rho(T)=1, \quad \rho\geq 0.\ee (we have employed the non-rigorous formulation \eqref{e:concintb}, \eqref{e:constrint} for the sake of heuristics). Here both $\rho(t)$ and $q(t)$ are scalar functions on $\Omega$. We have obtained a ballistic version of the geodesic problem for the Hellinger distance on the space of measures on $\Omega$, cf. \cite{Ay17,KLMP}, that can be made rigorous if we view $\rho(t)$ as an evolving measure on $\Omega$. This explains the title of this example, inspired by \cite{KM15}. We make the following ansatz: \be\label{e:rho} \rho(t)=(a(t-T)+1)^2,\ee where $a$ is a scalar function on $\Omega$.  Then $q(t)=2a(a(t-T)+1)$, and \eqref{e:concinfr} becomes \be\label{e:concians}   -\int_{\Omega} (v_0(1-Ta)^2+2a^2T)\,d\mu   \to \sup \ee  Furthermore, \eqref{e:concians} is optimized at $a=\frac {v_0}{Tv_0+2}.$ Moreover, an explicit calculation shows that the value of the corresponding functional \be\label{e:zz} \int_0^T\int_{\Omega} v_0 q\,d\mu \, dt  -\frac 1 2 \int_0^T\int_{\Omega} \rho^{-1}|q|^2\,d\mu \, dt,\ee cf. \eqref{e:concint}, coincides with the value in \eqref{e:iifr}. Hence, there is no duality gap, and $\rho$ defined by \eqref{e:rho} is optimal for \eqref{e:concinfr}. If \eqref{tv2} is violated, the functionals \eqref{e:concians} and \eqref{e:zz} become unbounded from above, and the corresponding $\mathcal J(v_0, T)$ becomes equal to $+\infty$. The bidual problem is \be\label{e:bidualfr}   -\int_{\Omega} \varphi(T)\,d\mu \to \inf \ee subject to the constraints \be\label{e:bdfr}\p_t \varphi+\frac 1 2 \varphi^2\leq 0, \quad \varphi(0)\leq v_0 \ \mathrm{a.e.}\ \mathrm{in}\ \Omega.  \ee If \eqref{tv2} holds, then the solution $v$ to the primal problem is an optimizer of the bidual problem. If not, then no $\varphi$ satisfies \eqref{e:bdfr}. \end{exmp}

\section{Existence} \label {s:exi}

In this section we show existence of solutions to the optimal ballistic transport  problem. Note that the conservativity \eqref{e:acons} and the approximation condition are not required for the results.  

We start with a relatively regular case that is determined by the so-called trace condition. If the formal conservativity \eqref{e:acons} holds, we revisit the original abstract Euler problem \eqref{e:aeuler} and define its generalized solutions produced from the constructed solutions of the dual problem. These generalized solutions automatically belong to the same regularity class as the strong solutions. 

\begin{definition}[Trace condition] \label{deftr} The operator $L$ is said to satisfy the trace condition if for any $\zeta\in D(L^*)\cap P(X^n)$ such that  the eigenvalues of the matrix $-L^* \zeta(x)$  are uniformly bounded from above by a constant $k$ for a.e. $x\in \Omega$, the eigenvalues of the matrix $L^* \zeta(x)$ are also uniformly bounded from above a.e. in $\Omega$ by a constant that depends only on $k$.\end{definition}
\begin{remark}[First facts about the trace condition] \label{r:tracel} The trace condition is particularly satisfied provided \be \label{e:invt} P L(qI)=0\ee for any $q\in X$ sufficiently smooth: it suffices to observe that the trace of $L^*\zeta$ vanishes almost everywhere. Indeed, $$ (\tr(L^*\zeta),q)=(L^*\zeta,qI)=(\zeta,PL(qI))$$ since $\zeta \in P(X^n)$. The trace condition holds for the incompressible Euler equation \eqref{e:euler} because $\mathcal P [-\di (qI)]=\mathcal P (-\nabla q)=0$, but not for the Hamilton-Jacobi equation \eqref{e:hjv} nor for the IVP \eqref{e:odeq}.  \end{remark}

\begin{theorem}[First existence theorem] \label{t:ex} Assume that $L$ satisfies the trace condition. Then for any $v_0\in P(X^n)$ there exists a maximizer $(E,B)$ to \eqref{e:conc} in the class \eqref{e:be}, and $0\leq \mathcal J(v_0, T)< +\infty$.  \end{theorem}
\begin{proof} It suffices to consider the pairs $(E,B)$ that meet the restrictions \eqref{e:constrweak}, \eqref{e:bwe}. Testing \eqref{e:conc} with $E=0,\;B=0$, we see that $\mathcal J(v_0)\geq 0$. Let $(E_m,B_m)$ be a maximizing sequence.  Since $0\le \mathcal J(v_0, T)$, without loss of generality we may assume that \be\label{e:ms}0 \le -\int_0^T (v_0,E_m)\, dt+\bigkm.\ee The eigenvalues of $-B_m$ are uniformly bounded from above because $I+2B_m\geq 0$. Since the trace condition is assumed, a uniform $L^\infty$ bound on $B_m$ follows directly from Definition \ref{deftr}. In other words, $I+2B_m\leq kI$ with some constant $k>0$ a.e. in $(0,T)\times\Omega$. By the definition of $\mathcal K$ in \eqref{e:defk1}, we have \be \bigkm\leq  \inf_{z\otimes z\le M}\int_0^T\left[(z,E_m)+\frac k 2(M, I)\right]\,dt=-\frac 1{2k}\int_0^T (E_m,E_m)\,dt.\ee We infer that 
\be \label{e:ms2} \frac 1{2k} \int_0^T(E_m,E_m)\leq -\int_0^T (v_0,E_m)\, dt\leq  2kTK_0+\frac 1{4k} \int_0^T(E_m,E_m),\ee which gives a uniform $L^2((0,T)\times \Omega;\R^n)$-bound on $E_m$. Moreover, by \eqref{e:ms2} the right-hand side of \eqref{e:ms} is uniformly bounded, whence $\mathcal J(v_0, T)<+\infty$. The functional $\mathcal K$ is concave and upper semicontinuous on $L^2((0,T)\times \Omega;\R^n)\times L^\infty((0,T)\times \Omega;\R^{n\times n}_s)$ as an infimum of affine continuous functionals, cf.  \eqref{e:defk1}. The functional $\int_0^T (v_0,\cdot)\, dt$ is a linear bounded functional on $L^2((0,T)\times \Omega;\R^n)$. Consequently, every weak-$*$ accumulation point of $(E_m,B_m)$ is a maximizer of \eqref{e:conc}. Note that the constraints \eqref{e:constrweak}, \eqref{e:bwe} are preserved by the limit. 
\end{proof}

\begin{remark}[Back to the original problem] \label{r:time} Assume \eqref{e:acons}.  Let $(E,B)$ be any maximizer of \eqref{e:conc}. Set $V:=-PE$. Formula \eqref{e:ta1}, in contrast to \eqref{e:vret}, does not rely on strict positive-definiteness of $I+2B$. We thus can define a generalized solution to \eqref{e:aeuler} by setting \be\label{e:geuler} v:=\frac 1 {T-t}\int_{t}^T V(s)\,ds\in AC^2_{loc}([0,T);X^n),\ee cf. \eqref{e:ta1}. Obviously, $v(t)\in P(X^n)$. Furthermore, $v$ automatically belongs to the same regularity class as the strong solutions. Indeed, $$\partial_t [(T-t)v]=-V\in L^2((0,T)\times \Omega; \R^n).$$ On the other hand, by Hardy's inequality, $v\in L^2((0,T)\times \Omega; \R^n)$. Finally, $v(t)\in D(L^*)$ for a.a. $t\in (0,T)$ and $$(T-t)L^*v\in L^\infty((0,T)\times \Omega; \R^{n\times n}_s).$$ Moreover, \begin{equation} \label{e:bandv} (T-t)L^*v=B.\end{equation} Indeed, let $\psi:[0,T]\to X^{n\times n}_s$ be an arbitrary sufficiently regular vector field, and set $\Psi(t):=\int_0^t \psi(\tau)\,d\tau$; integrating by parts  and using \eqref{e:constrweak}, we deduce that \begin{multline*} \int_0^T\left(B(t),\psi(t)\right)\,dt= -\int_0^T\left(PE(t),L\Psi(t)\right)\,dt=
		\\=-\int_0^T\left(\int_t^T PE(\tau)\,d\tau,L\psi(t)\right)\,dt=\int_0^T ((T-t)L^*v(t),\psi(t))\,dt.\end{multline*}
However, at this level of generality, $v$ is not necessarily a strong solution. In particular, it can violate the initial condition $v(0)=v_0$ or the first equality in \eqref{e:aeuler2}. We will return to this intriguing issue in Sections \ref{s:uniq} and \ref{KdV}.  The fact that the generalized solutions can violate the initial condition was already observed by Brenier \cite{CMP18} for the inviscid Burgers equation, although he did not employ  \eqref{e:geuler} but rather \eqref{e:vret}. By the way, the solutions of the problem from Remark \ref{convint} can violate the initial condition as well, cf. \cite{Br20}.  This implies that, generally speaking, the object defined by \eqref{e:geuler} or \eqref{e:vret} is not a solution of any of the types A, B or C as defined in the Introduction.   \end{remark} 

\begin{remark}[Brenier's formula \eqref{e:vret} vs. our formula \eqref{e:geuler}] \label{r:brenier} As we have just mentioned, there is an alternative way to define a generalized solution to \eqref{e:aeuler} (that goes back to \cite[Theorem 2.3]{CMP18}) by employing formula \eqref{e:vret}. Let us try to explain why in this paper we prefer  \eqref{e:geuler}. Firstly, the definition \be \label{e:vretn} v:=-(I+2B)^{-1}E\ee  is unambigous merely on the set $[I+2B>0]$, which is not obliged to be of full measure in $(0,T)\times \Omega$. In the scalar case $n=1$, it is possible to partially fix this issue by replacing \eqref{e:vretn} with the Radon-Nikodym derivative \begin{equation} \label{forb} v:=-\frac{\mathrm{d} E}{\mathrm{d}(I\mu\otimes dt+2B)},\end{equation} as is done in \cite{CMP18} for the inviscid Burgers equation (here $I\mu\otimes dt+2B$ and $E$ are viewed as measures, cf. our Theorem \ref{t:exweak} below and the discussion above it). But \eqref{forb} means that $v$ is well-defined by \eqref{forb} only in the support of the scalar measure $I\mu\otimes dt+2B$, and it is understandable from \cite[proof of Theorem 4.2]{CMP18} that  the above-mentioned support  fails to contain the whole set $[0,T]\times \Omega$ provided the characteristics intersect somewhere on that set (for the inviscid Burgers equation it is however possible to continue the  solution $v$ to $[0,T]\times \Omega$ by an ad hoc procedure unrelated to \eqref{e:vret}, cf. \cite[Proposition 4.1]{CMP18}). There is also an algebraic trick to make \eqref{e:vretn} work outside of the set $[I+2B>0]$ for any $n\in \mathbb{N}$. Namely, $v$ can be defined a.e. in $[0,T]\times \Omega$ as the unique vector in the column space of the matrix $I + 2B$ satisfying $v + 2B.v + E = 0$. Nevertheless, the latter definition is not compatible\footnote{In the case $n=1$ this definition would imply that $v=0$ a.e. in the set $[I + 2B=0]$. The generalized solution to the inviscid Burgers equation provided by \cite[Proposition 4.1]{CMP18} coincides with \eqref{forb} in the support of the measure $I\mu\otimes dt+2B$. However, generally speaking, it does not vanish outside of the support. This is easy to check by considering the initial datum $v_0(x)=sign \left(x-\frac 1 2\right),\ x\in [0,1]\simeq \mathbb{T}^1$ for which the solution of \cite[Proposition 4.1]{CMP18} can be computed by hand.} with the considerations of \cite[Section 4]{CMP18}. Anyway, even if the set $[I+2B>0]$ is of full measure in $[0,T]\times \Omega$, the regularity class for the generalized solution defined by \eqref{e:vretn} is blurred, and it is not clear whether the solution belongs to $P(X^n)$. 

That is why we decided to adopt formula \eqref{e:geuler}. This choice is crucial for the whole Section \ref{s:uniq}. Notably,  if the two definitions  \eqref{e:geuler} and \eqref{e:vretn} produce the same function $v$, then $v$ is a strong solution to \eqref{e:aeuler} with possibly different initial condition. This claim follows from our formulas \eqref{e:discrr} and \eqref{e:w1discr} below.  \end{remark}

As a matter of fact, we can prove existence of a less regular solution to the dual problem without assuming the trace condition.  However, in this case we need to slightly relax the definition of solutions. The next result requires us to assume that $\Omega$ is a separable locally compact metric space, and $\mu\in \mathcal M(\Omega)$. We restrict ourselves to the initial data $v_0\in C_0(\Omega)^n$. Then the optimal ballistic problem \eqref{e:conc} makes sense if $(E,B)$ is just a pair of measures from $\mathcal M((0,T)\times \Omega;\R^n)\times \mathcal M((0,T)\times \Omega;\R^{n\times n}_s)$. Indeed, \eqref{e:constrweak} is substituted with \be \label{e:constrweakr}\int_{(0,T)\times \Omega} dB:\p_t \Psi + \int_{(0,T)\times \Omega} dE\cdot PL \Psi=0\ee for all sufficiently smooth vector fields $\Psi: [0,T]\to X^{n\times n}_s$, $\Psi(0)= 0$;
\eqref{e:bwe} is replaced with \be\label{e:bwer} I\mu\otimes dt+2B\ \mathrm{is}\ \mathrm{a}\  \mathrm{positive-semidefinite}\ \mathrm{matrix-valued}\ \mathrm{measure};\ee 
\eqref{e:defk1} is replaced with \be\label{e:defk1r} \bigk:=\inf_{z\otimes z\le M}\left[\int_{(0,T)\times \Omega} z\cdot dE+ \frac 1 2  \int_{(0,T)\times \Omega} \tr M\,d \mu\otimes dt+ \int_{(0,T)\times \Omega} M : dB\right];\ee \eqref{e:conc} is substituted with \be\label{e:concrad}\mathcal J(v_0, T)=\sup_{E,B:\,\eqref{e:constrweakr},\eqref{e:bwer}}-\int_{(0,T)\times \Omega} v_0\cdot dE +\bigk.\ee 

\begin{theorem}[Second existence theorem] \label{t:exweak} Let $\Omega$ be a separable locally compact metric space and $\mu$ be a Borel probability measure on it.  Assume that \be \label{e:plo} PL(I)=0.\ee Then for any $v_0\in C_0(\Omega)^n$ there exists a maximizer $$(E,B)\in  \mathcal M((0,T)\times \Omega;\R^n)\times \mathcal M((0,T)\times \Omega;\R^{n\times n}_s)$$ of \eqref{e:concrad}, and $0\leq \mathcal J(v_0, T)<+\infty$.  \end{theorem}
\begin{proof} The proof is somewhat similar to the proof of Theorem \ref{t:ex} but a priori bounds are obtained in a different manner. It suffices to consider the pairs $(E,B)$ that meet the restrictions \eqref{e:constrweakr}, \eqref{e:bwer}. Testing \eqref{e:defk1r} with $z=0,\ M=0$, we see that \be\label{e:estk} \bigk\leq 0.\ee Analogously, testing \eqref{e:concrad} with $E=0,\;B=0$, we infer that $\mathcal J(v_0, T)\geq 0$.

 We now claim that, for any $v_0\in C_0(\Omega)^n$ and any pair $(E,B)$ satisfying  \eqref{e:constrweakr}, \eqref{e:bwer}, one has  \be\label{e:hold1}\left( \int_{(0,T)\times \Omega} v_0\cdot dE \right)^2\leq -2\bigk\cdot\int_{(0,T)\times \Omega} d(I\mu\otimes dt+2B) v_0:v_0,\ee and therefore  \be\label{e:hold2}\left( \int_{(0,T)\times \Omega} |dE|\right)^2\leq -2\bigk\cdot\, \tr\int_{(0,T)\times \Omega} d(I\mu\otimes dt+2B).\ee Indeed, it is clear that $dE_0:= -d(I\mu\otimes dt+2B). v_0$ is a finite vectorial Radon measure on $(0,T)\times \Omega$. Using the results of \cite{G66}, one can check that $-2\bigk$ defined via \eqref{e:defk1r} is a positive-semidefinite quadratic form w.r.t. $E$ for fixed $B$. Let $\omega(\cdot,\cdot)$ be the corresponding bilinear form obtained by polarization. By the Cauchy-Schwarz inequality, \be\label{e:hold3}\omega^2(E,E_0)\leq 4\bigk\mathcal K(E_0,B).\ee The infimum in \eqref{e:defk1r} with $E=E_0$ is achieved at $z=v_0$, whence $$\mathcal K(E_0,B)=-\frac 1 2 \int_{(0,T)\times \Omega} d(I\mu\otimes dt+2B) v_0:v_0.$$ Let $z_m$ be the minimizing sequence for the infimum in \eqref{e:defk1r} with $E$ replaced by $E-E_0$. A direct calculation verifies that  \begin{multline*}2\int_{(0,T)\times \Omega} v_0\cdot dE \\+ \left[\int_{(0,T)\times \Omega} z_m\cdot d(E-E_0)+ \frac 1 2  \int_{(0,T)\times \Omega} \tr(z_m\otimes z_m) \,d \mu\otimes dt+ \int_{(0,T)\times \Omega} (z_m\otimes z_m) : dB\right]\\ = \left[\int_{(0,T)\times \Omega} \tilde z_m\cdot d(E+E_0)+  \frac 1 2  \int_{(0,T)\times \Omega} \tr(\tilde z_m\otimes \tilde z_m) \,d \mu\otimes dt+ \int_{(0,T)\times \Omega} (\tilde z_m\otimes \tilde z_m) : dB\right]\\ \geq \mathcal K(E+E_0,B). \end{multline*} where $\tilde z_m=z_m+2 v_0$.  Consequently, \begin{multline*}\int_{(0,T)\times \Omega} v_0\cdot dE \geq \frac 12 \mathcal K(E+E_0,B)- \frac 12 \mathcal K(E-E_0,B)\\=\frac 1 4 \left(\omega(E-E_0,E-E_0)-\omega(E+E_0,E+E_0)\right)=-\omega(E,E_0).\end{multline*} The reverse inequality is proved in a similar fashion, swapping the roles of $E+E_0$ and $E-E_0$. Hence, the left-hand sides of \eqref{e:hold1} and \eqref{e:hold3} coincide, but due to the observations above so do the right-hand sides. 
 
 Since we have proved \eqref{e:hold1} for arbitrary $v_0\in C_0(\Omega)^n$, inequality \eqref{e:hold2} immediately follows. 

Let $(E_m,B_m)$ be a maximizing sequence for \eqref{e:concrad}. Without loss of generality, it satisfies \be \label{e:intrm1} 0\le -\int_{(0,T)\times \Omega} v_0\cdot dE_m +\bigkm\ee Observe that, due to \eqref{e:plo}, $$\tr\int_{(0,T)\times \Omega} dB_m=0$$  (indeed, when $\Omega$ is compact it suffices to test \eqref{e:constrweakr} with $tI$; in the general case we can test with $t\Psi_\alpha I$ where $\{\Psi_\alpha\}$ is a compactly supported partition of unity on $\Omega$, and sum up the results). Since $I\mu\otimes dt+2B$ is a positive-semidefinite matrix-valued measure, we have got a uniform $\mathcal M((0,T)\times \Omega;\R^{n\times n}_s)$-bound on $B_m$.   By \eqref{e:intrm1}  and \eqref{e:hold1}, \be \label{e:intrm2} 0\le \sqrt{-2nT\|v_0\|^2_{C_0(\Omega)} \bigkm}+\bigkm.\ee 
By some elementary algebra and \eqref{e:hold2}, this implies a uniform $\mathcal M((0,T)\times \Omega;\R^n)$-bound on $E_m$. Hence, the right-hand side of \eqref{e:intrm1} is uniformly bounded, which gives $\mathcal J(v_0, T)<+\infty$. The right-hand side of \eqref{e:concrad} is concave and upper semicontinuous on $\mathcal M((0,T)\times \Omega;\R^n)\times \mathcal M((0,T)\times \Omega;\R^{n\times n}_s)$ as an infimum of affine continuous functionals. Consequently, every weak-$*$ accumulation point of $(E_m,B_m)$ is a maximizer of \eqref{e:concrad}. 
\end{proof}

\begin{remark}[Relevance of \eqref{e:plo}] \label{plusin} Assumption \eqref{e:plo} is essential since it prevents the dumb blow-up $\mathcal J(v_0, T)=+\infty$ as in Example 
\ref{ex:h}. On the other hand, \eqref{e:plo} is a very light assumption, since it holds, for instance, for any differential operator $L$ without zero-order terms. Although \eqref{e:plo} does not necessarily follow from the trace condition, it is obviously implied by the stronger assumption \eqref{e:invt}. 
\end{remark}

\section{Energy levels and weak-strong uniqueness} \label{s:uniq}

As we already discussed, although the energy is conserved for the strong solutions of the abstract Euler equation,  this feature can fail for weak solutions. Moreover, in view of the results of \cite{W11}, it is plausible that the optimal values $\mathcal I(v_0,T) $ and $\mathcal J(v_0,T)$ may exceed the reference value $TK_0$.  In this section we show that this  cannot happen for $\mathcal J(v_0,T)$, at least when the dual problem has a solution in the class \eqref{e:be} (which is  particularly true if the trace condition holds). It will be also proved that if $\mathcal J(v_0,T)$ assumes the critical value $TK_0$, the generalized solution to \eqref{e:aeuler} constructed in Remark \ref{r:time} is automatically a strong solution with the same initial datum. This will lead to ``weak-strong uniqueness'' of the maximizer of the dual problem. The results of this section are new even for the incompressible Euler.

\begin{theorem}[Sharp upper bound for $\mathcal J$] \label{t:ldistr} Assume \eqref{e:acons} and that the approximation condition holds for $L$. Fix $v_0\in P(X^n)$ and $T>0$. Assume that there exists a maximizer $(E,B)$ to \eqref{e:conc} in the class \eqref{e:be}. Then $\mathcal J(v_0,T)\leq TK_0$.  In addition, if $\mathcal J(v_0,T)=TK_0$, then the generalized solution $v$ to \eqref{e:aeuler} defined by \eqref{e:geuler} is necessarily a strong solution to \eqref{e:aeuler} with the same initial datum $v_0$.  
\end{theorem}

\begin{proof} 
	Suppose that $\mathcal J(v_0,T)\geq TK_0$ (otherwise there is nothing to prove). Set \be \label{e:discrr} r:=v+2B.v+E\in L^2((0,T)\times \Omega;\R^n).\ee

	We claim that \be \label{e:w1discr}\p_t v+2P[L^*v.v]=P\frac r{T-t}\ee a.e. in $(0,T)\times \Omega$.  Indeed, using \eqref{e:bandv}, we compute \begin{multline}Pr=P(v+2B.v+E)=v+2(T-t)PL^*v.v+(-v+(T-t)\p_t v)\\ = (T-t)\p_t v+2(T-t)P[L^*v.v].\end{multline}
	
As in the proof  of Proposition \ref{propap}, the approximation condition implies that \be \label{e:w3++}(v\otimes v, L^*v)=0\ee  a.e. in $(0,T)$.  Remembering 
	\eqref{e:bandv}, we see that \be \label{e:w3+++}\int_0^T (v\otimes v, B)\, dt=0.\ee 
	
Taking the $X^n$ scalar product of \eqref{e:w1discr} with $(T-t)v$ for a.a. $t\in (0,T)$ and employing \eqref{e:w3++} we get \begin{equation} \label{e:withk1} (T-t)\frac d {dt} K(t)= (r,v)\in L^1(0,T).\end{equation}
Integrating by parts, we obtain \begin{equation} \label{e:withk2}\int_0^T K(t)\,dt-TK(0)=\int_0^T (r,v)\, dt.\end{equation}  
Employing \eqref{e:vab} with $a(t)=(T-t)v$, which is still valid at this level of generality, we get \be \label{e:vabn}-\int_0^T (v_0,E)\, dt=T(v_0,v(0)). \ee

Let $\omega$ be the bilinear form defined as in the proof of Theorem \ref{t:exweak}.
Using \eqref{e:discrr}, \eqref{e:w3+++}, \eqref{e:withk2} and \eqref{e:vabn}, we compute \begin{multline*} \mathcal J(v_0,T)= -\int_0^T (v_0,E)\, dt+ \bigk=-\int_0^T (v_0,E)\, dt-\frac 1 2 \omega(E,E)\\=T(v_0,v(0))-\frac 1 2 \omega(r,r)-\frac 1 2 \omega((I+2B)v,(I+2B)v)+\omega(r,(I+2B)v)\\=T(v_0,v(0))+\mathcal K(r,B)+\mathcal K((I+2B)v,B)+\omega(r,(I+2B)v)
	\\=T(v_0,v(0))+\mathcal K(r,B)-\frac 12 \int_0^T ((I+2B)v,v)\,dt+\int_0^T(r,v)\,dt\\
	=T(v_0,v(0))+\mathcal K(r,B)+\int_0^T(r,v)\,dt-\int_0^T K(t)\,dt \\
		=T(v_0,v(0))-TK(0)+\inf_{z\otimes z\le M}\int_0^T \left[(z,r)+\frac 1 2(M, I+2B)\right]\, dt. \end{multline*} Hence, $$\mathcal J(v_0,T)-TK_0+\frac 12T(v_0-v(0), v_0-v(0))=\inf_{z\otimes z\le M}\int_0^T \left[(z,r)+\frac 1 2(M, I+2B)\right]\, dt.$$ Since the left-hand side is non-negative and the right-hand side is non-positive, both sides vanish, whence $v(0)=v_0$ and $\mathcal J(v_0,T)= TK_0$.  Since $B\in L^\infty((0,T)\times \Omega;\R^{n\times n}_s)$, there is $k>0$ such that $I+2B\leq kI$ a.e. in $(0,T)\times \Omega$. Consequently, \begin{equation} \label{e:r00}
		0=\inf_{z\otimes z\le M}\int_0^T \left[(z,r)+\frac 1 2(M, I+2B)\right]\, dt\leq \inf_{z\otimes z\le M}\int_0^T \left[(z,r)+\frac k 2\tr M \right]\, dt\leq 0.\end{equation} Some simple algebra shows that \eqref{e:r00} implies $r=0$ a.e. in $(0,T)\times \Omega$. Now \eqref{e:w1discr} yields \eqref{e:aeuler2}. 
\end{proof}

In order to get a ``weak-strong uniqueness'' statement we first need to show that the strong solutions are unique. 

\begin{lemma}[Uniqueness of strong solutions]\label{l.uniques} Assume \eqref{e:acons} and that the approximation condition holds for $L$. Let $v^1$, $v^2$ be two strong solutions to  \eqref{e:aeuler} with the same initial datum $v_0\in P(X^n)$. Then $v^1(\tau)=v^2(\tau)$ for all $\tau\in [0,T)$.
\end{lemma}

\begin{proof} Denote $w:=v^1-v^2$. Fix $\tau\in [0,T)$, and let $h:(0,\tau)\to \R$ be a smooth non-negative compactly supported scalar function. Consider the first equality in \eqref{e:aeuler2} both for $v^1$ and $v^2$, compute the difference thereof, and calculate the scalar product of the resulting equality with $wh$ in $L^2((0,\tau)\times \Omega, \R^n)$ to obtain
$$\int_0^\tau \left[-\frac 1 2 h'(t)(w(t),w(t)) +2h(t)\left([L^*v^1(t).w(t)]+[L^*w(t).v^2(t)],w(t)\right)\right]\,dt=0.$$ Using \eqref{polar}, this can be rewritten as 

\begin{equation}\label{e:410}\int_0^\tau \left[-\frac 1 2 h'(t)(w(t),w(t)) \right]\,dt = \int_0^\tau h(t)\left[-2\left(w(t), L^*v^1(t)\otimes w(t)\right)+\left(L^*v^2(t), w(t)\otimes w(t)\right)\right]\,dt.\end{equation}
This is not completely rigorous because \eqref{polar} is valid for sufficiently smooth functions.  However, employing the approximation condition, more or less as we did during the proof of Proposition \ref{propap}, we can justify the applicability of \eqref{polar} provided $v^1$, $v^2$ merely belong to the regularity class \eqref{e:strongclass}. 

It follows from \eqref{e:410} that \begin{equation*}\int_0^\tau \left[- h'(t)(w(t),w(t)) \right]\,dt \leq C \int_0^\tau h(t)\left(w(t), w(t)\right)\,dt\end{equation*} with $C$ depending on the $L^\infty$ norms of $v^1$ and $v^2$. 
By Gr\"onwall-Bellman lemma, $$(w(\tau),w(\tau))\leq e^{C\tau} (w(0),w(0))=0.$$
\end{proof}

\begin{corollary}[Weak-strong uniqueness] \label{c:wsu} Assume \eqref{e:acons} and that the approximation condition holds for $L$. If $\mathcal J(v_0,T)=TK_0$, then there is at most one maximizer of \eqref{e:conc} in the class \eqref{e:be}. In particular, in the framework of Theorem \ref{t:smooth} the pair $(E_+,B_+)$ is the unique  maximizer of \eqref{e:conc} in the class \eqref{e:be}. 
\end{corollary}

\begin{proof} By Theorem \ref{t:ldistr}, any alleged maximizer $(E,B)$ determines a strong solution $v$ to \eqref{e:aeuler}. By Lemma \ref{l.uniques}, all such strong solutions coincide. The first claim follows from \eqref{e:bandv} and \eqref{e:discrr} (with $r\equiv 0$, cf. the proof of Theorem \ref{t:ldistr}).  It remains to remember that in the framework of Theorem \ref{t:smooth} one has $\mathcal J(v_0,T)=TK_0$.
\end{proof}

\begin{corollary}[Necessary and sufficient conditions for a discrepancy] \label{c:enlevels} Assume \eqref{e:acons} and that the approximation condition hold for $L$. Assume that there exist a strong solution $u$ to \eqref{e:aeuler} with the initial datum $v_0$ and a maximizer $(E,B)$ to \eqref{e:conc} in the class \eqref{e:be}. Let $v$ be the generalized solution to \eqref{e:aeuler} defined by \eqref{e:geuler}. Then the following three conditions are equivalent:

i) $v$ does not coincide with $u$ on $[0,T)$;

ii) $\mathcal J(v_0,T)< TK_0$;

iii) $I\not\geq 2(t-T)L^*u(t)$ on a set of positive measure in $(0,T)\times \Omega$.

Moreover, in all these cases, $v$ violates the initial condition and/or the first equality in \eqref{e:aeuler2}. 
\end{corollary}

\begin{proof} The implication i)$\to$ ii) is an immediate consequence of Theorem \ref{t:ldistr}
and Lemma \ref{l.uniques}. The implication ii) $\to$ iii) follows from Theorem \ref{t:smooth}. Finally, assume that iii) holds and i) is violated. Then $u=v$ on $[0,T)$, and from \eqref{e:bwe} and \eqref{e:bandv} we see that $I+ 2(T-t)L^*u(t)\geq 0$, which contradicts iii).   \end{proof}

\begin{remark} Note that the framework of Corollary \ref{c:enlevels} does not necessarily imply a duality gap: the situation $\mathcal I=\mathcal J<TK_0$ is hypothetically possible.  \end{remark}

We complete Section \ref{s:uniq} as well as the whole ``abstract'' part of the article with the following useful observation.  
\begin{remark}[Introducing a linear term into the original problem] \label{r:aoper}  The theory in this paper can be adapted to  the setting \be \label{e:aeulerd}\p_t v+PAv= PL(v\otimes v)
, \quad v(t,\cdot)\in P(X^n), \quad v(0,\cdot)=v_0\in P(X^n)\ee where $$P:X^n\to X^n$$ is any orthogonal projector, and $$L: D(L)\subset X_s^{n\times n}\to X^n,\quad A: D(A)\subset X^n\to X^n$$ are linear operators, satisfying (optionally, as for \eqref{e:acons} above) \be \label{e:aconsd} (L(v\otimes v), v)=(Av,v)=0, \quad v\in P(X^n),\ee for any  sufficiently smooth vector field $v$.  

Indeed, we set $$\tilde v_0=(v_0,1)\in X^{n+1}\simeq X^n \times X,$$
$$\tilde v=(v,1):[0,T]\to X^{n+1},$$ $$\tilde P:X^{n+1}\to X^{n+1},\ \tilde P(v,q)=\left(Pv,\int_\Omega q\,d\mu\right),$$ $$\tilde L:D(\tilde L)\subset X_s^{(n+1)\times(n+1)}\to X^{n+1},\ D(\tilde L)=\left(\begin{array}{@{}c|c@{}}
D(L)& D(A) \\  \hline 
D(A)^\top & X 
\end{array}\right),$$ $$\tilde L\left(\begin{array}{@{}c|c@{}}
M& \upsilon \\  \hline 
\upsilon^\top & q 
\end{array}\right)=\left(\begin{array}{@{}c@{}}
LM-A\upsilon \\  \hline 
0 
\end{array}\right).$$ Tautologically, \be\label{e:taut} \p_t 1=0. \ee The ``system'' \eqref{e:aeulerd}, \eqref{e:taut} can be recast as  \be\label{e:aeulerdd}\p_t \tilde v= \tilde P\tilde L(\tilde v\otimes \tilde v), \quad \tilde v(t,\cdot)\in \tilde P(X^{n+1}), \quad \tilde v(0,\cdot)=\tilde v_0\in \tilde  P(X^{n+1}),\ee which has the structure of \eqref{e:aeuler}. Moreover, any $\tilde v\in \tilde P(X^{n+1})$ can be expressed as $$(v,a)\simeq \left(\begin{array}{@{}c@{}}
v \\  \hline 
a 
\end{array}\right),\quad v\in P(X^{n}),\quad  a=cst.$$ 

If \eqref{e:aconsd} is assumed,  \be \label{e:aconsdd} \left(\tilde L\left(\tilde v \otimes \tilde v\right), \tilde v\right)=\left(\tilde L\left(\left(\begin{array}{@{}c@{}}
v \\  \hline 
a 
\end{array}\right)\otimes \left(\begin{array}{@{}c@{}}
v \\  \hline 
a 
\end{array}\right)\right), \left(\begin{array}{@{}c@{}}
v \\  \hline 
a 
\end{array}\right)\right)=\left(\left(\begin{array}{@{}c@{}}
L(v\otimes v)-aAv \\  \hline 
0
\end{array}\right), \left(\begin{array}{@{}c@{}}
v \\  \hline 
a 
\end{array}\right)\right)=0,\ee i.e., condition \eqref{e:acons} is met for the extended problem. If \eqref{e:invt} holds for $L$, it is valid for $\tilde L$ as well. Indeed, in this situation  we have$$\tilde P\tilde L\left(\begin{array}{@{}c|c@{}}
qI & 0 \\  \hline 
0 & q 
\end{array}\right)=\tilde P\left(\begin{array}{@{}c@{}}
L(qI) \\  \hline 
0 
\end{array}\right)=0$$ for any $q\in X$ sufficiently smooth. As concerns the validity of the approximation condition, we believe that  it is more appropriate to check it directly for $\tilde L$ when applying the theory to a particular PDE rather than to formulate a generic criterion in terms of $L$ and $A$.  \end{remark}


\section{Applications to conservative PDE} \label{Sec3}

The purpose of this section is to prove that the ideal incompressible MHD, the multidimensional Camassa-Holm, EPDiff, Euler-$\alpha$, KdV and Zakharov-Kuznetsov equations, and two models of motion of incompressible elastic fluids satisfy the trace condition and thus meet the requirements of Theorem \ref{t:ex}. We also show that Theorem \ref{t:exweak} is applicable to the template matching equation. All these equations satisfy the approximation condition and either \eqref{e:acons} or \eqref{e:aconsd}, so Theorem \ref{t:smooth} and the results of Section \ref{s:uniq} apply to all of them (with some restrictions in the case of the template matching equation). 

\subsection*{Additional notation and conventions} 
To fix the ideas, in this section we restrict ourselves to the case of the periodic box $\Omega=\mathbb T^d$. Of course, it is also possible to consider an open domain $\Omega\subset \R^d$ of finite volume (in this case $\mu$ would be the normalized Lebesgue measure). In this setting the classes $\mathcal R$ and $\widehat{\mathcal R}$ of \emph{sufficiently smooth} functions can be defined in a conventional way. 

The symbol $\mathcal P$ denotes the Leray-Helmholtz projector in $X^d$. 

The symbol $\R^{(n\times n)\times(n\times n)}$ denotes the space of matrices with matricial entries. For a tensor $\Xi\in \R^{(n\times n)\times(n\times n)}$, define the matrices $\widehat\Xi, \widecheck\Xi\in \R^{n\times n}$ by $$\widehat\Xi_{ij}=\sum_{k} \Xi_{ik,jk},\quad \widecheck\Xi_{ij}=\sum_{k} \Xi_{ki,kj}.$$ For a matrix $M\in \R^{n\times n}$, define the tensors $\widehat M^*, \widecheck M^*\in \R^{(n\times n)\times(n\times n)}$ by $$\widehat M^*_{ik,jl}=M_{ij}\delta_{kl},\quad \widecheck M^*_{ik,jl}=M_{kl}\delta_{ij}.$$ For a tensor $\Upsilon\in \R^{n\times n \times n}$, denote $$\widetilde \Upsilon_{ijk}:=\Upsilon_{ikj}.$$ 
\subsection{Ideal incompressible MHD} The  ideal incompressible MHD equations \cite{Ar98} read
\begin{gather} \label{mhd1} \p_t u+\di (u\otimes u)+\nabla p=\di (b\otimes b),\\ \label{mhd2} 
\p_t b+\di (b\otimes u)=\di (u\otimes b),\\
\label{mhd3}  \di u=0,\\
\label{mhd4}  \di b=0,\\ \label{mhd5} 
u(0)=u_0,\quad b(0)=b_0.\end{gather} The unknowns are $u,b:[0,T]\times \Omega\to \R^d$ and $p:[0,T]\times \Omega\to \R$.  The ideal incompressible MHD equations are the geodesic equations (in our initial-value problem setting  it is more correct to call them the exponential map equations) on the semidirect product of the Lie group of volume-preserving diffeomorphisms with the dual of its Lie algebra \cite{Ar98}.    Since  $$\di\di (b\otimes u)=\di\di (u\otimes b),$$ we can rewrite \eqref{mhd1}, \eqref{mhd2} in the equivalent form \begin{gather} \label{mhd11} \p_t u=\mathcal P(\di (b\otimes b)-\di (u\otimes u)),\\ \label{mhd21} 
\p_t b=\mathcal P(\di (u\otimes b)-\di (b\otimes u)).\end{gather} Set $$n=2d,\ v=(u,b): [0,T]\to X^n\simeq X^d\times X^d,$$ $$P: X^n\to X^n,\quad P(\upsilon,\beta)=(\mathcal P \upsilon, \mathcal P \beta),$$ $$L: D(L)\subset X_s^{n\times n}\to X^n, \quad L\,\left(\begin{array}{@{}c|c@{}}
M& N \\  \hline 
N^\top & S 
\end{array}\right)=\left(\begin{array}{@{}c@{}}
 \di S-\di M \\ \hline
\di N-\di (N^\top)
\end{array}\right).$$ Then \eqref{mhd1}--\eqref{mhd5} becomes the quadratic problem \eqref{e:aeuler}. It is straightforward to check that \eqref{e:acons} holds for $v=(u,b)$ sufficiently smooth. Let $q\in X$ be a sufficiently smooth function. Then $$PL\left(\begin{array}{@{}c|c@{}}
qI& 0 \\  \hline 
0& qI 
\end{array}\right)=P\left(\begin{array}{@{}c@{}}
\nabla q-\nabla q \\  \hline 
0
\end{array}\right)=0.$$ In view of Remark \ref{r:tracel}, Theorem \ref{t:ex} and Remark \ref{r:time} are applicable, and we get \begin{corollary} For any $(u_0,b_0)\in X^d\times X^d$ with $\di u_0=\di b_0=0$, there exists a generalized solution \eqref{e:geuler} to \eqref{mhd1}--\eqref{mhd5}. \end{corollary}
Note that the restriction \eqref{e:goodv} related to the bidual problem holds provided both $u_0$ and $b_0$ have zero mean (as vector fields). 

\subsection{Multidimensional Camassa-Holm} The $H(\di)$ geodesic equation (sometimes called the multidimensional Camassa-Holm system)  \cite{KS01,GV18} looks like \begin{gather} \label{ch1} \p_t m+(\nabla u)^{\top}.m+\di (m\otimes u)=0,\\
\label{ch2} m=u-\nabla \di u,\\
\label{ch3} u(0)=u_0.\end{gather}
The unknown is $u:[0,T]\times \Omega\to \R^d$. It describes the geodesics of the diffeomorphism group with $H^1_{\di}$ metric, see, e.g., \cite{KLMP}. Another geodesic interpretation was discussed in \cite{GV18}. Apart from that, the $H(\di)$ geodesic equation can be embedded into the incompressible Euler equation on a specifically constructed higher-dimensional manifold \cite{N19,GV18}. We recall (cf. \cite{KLMP,B91}) that, loosely speaking, there is a ``fiber-base'' duality between the Monge-Kantorovich transport \cite{villani03topics} and Euler's equations \eqref{euler1}-\eqref{euler4}. In a similar way, one can think, cf. \cite{GV18}, about a ``fiber-base'' duality between \eqref{ch1}--\eqref{ch3} and the unbalanced optimal transport \cite{KMV16A,CP18,LMS16}.

We now define the relevant projector. Namely, for each $(\upsilon,\sigma)\in X^{d+1}\simeq X^d\times X$, we consider its orthogonal projection over the vector fields of the form $(u,\di u)$. This is related to the ``duality'' above and to the unbalanced version of Brenier's polar factorization theorem \cite{B91} that was discussed in preliminary preprint versions of \cite{GV18}. The explicit expression of the projector is \be P:X^{d+1}\to X^{d+1}, \quad P\left(\begin{array}{@{}c@{}}
\upsilon \\  \hline 
\sigma
\end{array}\right)=\mathcal P_{\di}\left(\begin{array}{@{}c@{}}
\upsilon \\  \hline 
\sigma
\end{array}\right):=\left(\begin{array}{@{}c@{}}
\upsilon-\nabla (I-\Delta)^{-1}(\sigma-\di\upsilon) \\  \hline 
\sigma-(I-\Delta)^{-1}(\sigma-\di\upsilon)
\end{array}\right).\ee

Set $$n=d+1,\ v=(u,\di u): [0,T]\to X^n\simeq X^d\times X,$$  $$L: D(L)\subset X_s^{n\times n}\to X^n, \quad L\,\left(\begin{array}{@{}c|c@{}}
M& \upsilon \\  \hline 
\upsilon^\top & q
\end{array}\right)=\left(\begin{array}{@{}c@{}}
 -\di M \\ \hline
-\di \upsilon+\frac 1 2 \tr M+\frac 1 2 q
\end{array}\right).$$  

We claim that the Camassa-Holm system \eqref{ch1}--\eqref{ch3} is tantamount to the abstract Euler equation \eqref{e:aeuler} with $P$ and $L$ just defined. Indeed, denote $p:=\di u$, $p_0:=\di u_0$ in \eqref{ch1}--\eqref{ch3}. After some calculations, one finds that \eqref{ch1}--\eqref{ch3} is equivalent to \begin{gather} \label{ch4} \p_t u=-\di (u\otimes u)+\nabla\left[\p_t p+\di (up)-\frac 1 2 |u|^2-\frac 1 2 p^2\right],\\
\label{ch5} p=\di u,\\
\label{ch6} u(0)=u_0,\, p(0)=p_0.\end{gather}
Tautologically, 
\be \label{ch7} \p_t p= -\di (up)+\frac 1 2 |u|^2+\frac 1 2 p^2+\left[\p_t p+\di (up)-\frac 1 2 |u|^2-\frac 1 2 p^2\right].\ee
The system \eqref{ch4}--\eqref{ch7} can be rewritten as
\be  \label{e:aeulerch}\p_t v= L(v\otimes v)+\left(\begin{array}{@{}c@{}}
 \nabla \xi \\ \hline
\xi
\end{array}\right),\quad v(0)=v_0,\ee where $$v(t)=\left(\begin{array}{@{}c@{}}
 u(t) \\ \hline
p(t)
\end{array}\right)\in P(X^n)$$ and \be \label{e:xi} \xi=\p_t p+\di (up)-\frac 1 2 |u|^2-\frac 1 2 p^2.\ee Applying the projector $P$ to both sides of \eqref{e:aeulerch}, we get \eqref{e:aeuler}. Reciprocally, \eqref{e:aeuler} implies \eqref{e:aeulerch} where $\xi$ necessarily satisfies \eqref{e:xi} due to \eqref{ch7}.

A not very tedious calculation verifies \eqref{e:acons} for $v=(u,\di u)$ sufficiently smooth. However, $$PL\left(\begin{array}{@{}c|c@{}}
qI& 0 \\  \hline 
0 & q
\end{array}\right)=P\left(\begin{array}{@{}c@{}}
 -\nabla q \\ \hline
\frac {d+1} 2 q
\end{array}\right),$$ which yields that the requirement \eqref{e:invt} is not met, and we need to find another way to secure the trace condition. It will be based on the following simple multidimensional variant of the Gr\"onwall-Bellman lemma. \begin{lemma} \label{l:gr} Consider a function $\psi\in W^{1,1}(\mathbb T^d)$ such that a.e. in $\mathbb T^d$ one has \be \label{e:ghyp} |\nabla \psi(x)|\le c\psi(x)\ee  with a constant $c$. Then $\psi\in C(\mathbb T^d)$, and \be \label{e:gcons} |\psi(x)|\leq e^{\frac {c\sqrt d}2 } \int_{\mathbb T^d} \psi(y)\,dy, \quad x\in \mathbb T^d. \ee \begin{proof} By Sobolev embedding, $\psi\in L^p(\mathbb T^d)$, $1 -\frac 1 n=\frac 1 p$, whence $\psi\in W^{1,p}(\mathbb T^d)$. Bootstrapping, we derive that $\psi\in W^{1,\infty}(\mathbb T^d)\subset C(\mathbb T^d)$. Consequently, $\log \psi\in W^{1,\infty}(\mathbb T^d)$ because \be \label{e:ghypl} |\nabla \log \psi(x)|\le c\ee due to \eqref{e:ghyp}. 
Since $|\mathbb T^d|=1$, there is $x^0\in \mathbb T^d$ such that $\psi(x^0)=\int_{\mathbb T^d} \psi(y)\,dy$. By \eqref{e:ghypl}, \be \label{e:ghyp1} |\log \psi(x)-\log \psi(x^0)|\le c|x-x^0|\leq \frac {c\sqrt d}2, \quad x\in \mathbb T^d,\ee which implies \eqref{e:gcons}. \end{proof}\end{lemma}

We return to the Camassa-Holm system. The adjoint operator is

$$L^*: D(L^*)\subset X^n\to X_s^{n\times n}, \quad L^*\,\left(\begin{array}{@{}c@{}}
 \phi\\ \hline
\chi
\end{array}\right)=\frac 1 2 \left(\begin{array}{@{}c|c@{}}
\nabla \phi +(\nabla \phi)^\top+\chi I& \nabla \chi \\  \hline 
(\nabla \chi)^\top & \chi
\end{array}\right).$$  If $(\phi,\chi) \in P(X^n)$, then $\chi=\di \phi$. If the eigenvalues of \be \label{e:mat} \left(\begin{array}{@{}c|c@{}}
\nabla \phi +(\nabla \phi)^\top+\chi I& \nabla \chi \\  \hline 
(\nabla \chi)^\top & \chi
\end{array}\right)(x),\quad \chi=\di \phi,\ee are bounded from below, there is $k\ge 0$ such that $$\left(\begin{array}{@{}c|c@{}}
\nabla \phi +(\nabla \phi)^\top+(\chi+k) I& \nabla \chi \\  \hline 
(\nabla \chi)^\top & \chi+k
\end{array}\right)(x)\geq 0.$$ 
In particular, $\chi+k\geq 0$. Moreover,
considering the principal minors of order $2$, we see that $$(\chi+k+2\p_{x_i}\phi_i)(\chi+k) \geq (\p_{x_i} \chi)^2.$$ Thus, $$3(\chi+k)^2=(3\chi+3k)(\chi+k)\geq (3\chi+k)(\chi+k)\geq  |\nabla \chi|^2.$$ Since $\int_{\mathbb T^d} \chi(y)\,dy=0$, Lemma \ref{l:gr} implies that \be \label{e:gx} \chi(x)+ k \leq k e^{\frac {\sqrt {3d}} 2} , \quad x\in \mathbb T^d. \ee This provides a uniform bound on the trace of the matrix in \eqref{e:mat}. Hence, the eigenvalues of this matrix are bounded from above, and the trace condition holds. We infer
 \begin{corollary} For every $u_0\in X^d$, there exists a generalized solution \eqref{e:geuler} to \eqref{ch1}--\eqref{ch3}. \end{corollary}
 
 Note that \eqref{e:goodv} holds provided $u_0$ has zero mean.

\subsection{EPDiff} The EPDiff equations \cite{EPD05,EPD09,YO10,EPD13,MM13,EPD16} are
\begin{gather} \label{epd1} \p_t m+(\nabla u)^{\top}.m+\di (m\otimes u)=0,\\ \label{epd2}
m=u-\Delta u,\\ \label{epd3}
u(0)=u_0.\end{gather} The unknown is $u:[0,T]\times \Omega\to \R^d$. The EPDiff equations are the Euler-Arnold equations on the diffeomorphism group with $H^1$ metric, see, e.g., \cite{KLMP}. 

For each $(\upsilon,M)\in X^{d(1+d)}\simeq X^d\times X^{d\times d}$, we consider its orthogonal projection over the fields of the form $(u,\nabla u)$. This is related to a (non-ballistic) matricial optimal transport problem, cf. \cite[Remark C1]{BV18}. The explicit expression of the projector is \be \label{e:blo} P:X^{d(1+d)}\to X^{d(1+d)}, \quad \quad P\left(\begin{array}{@{}c@{}}
\upsilon \\  \hline 
M
\end{array}\right)=\mathcal P_{\nabla}\left(\begin{array}{@{}c@{}}
\upsilon \\  \hline 
M
\end{array}\right):=\left(\begin{array}{@{}c@{}}
\upsilon-\di (I-\nabla \di)^{-1}(M-\nabla\upsilon) \\  \hline 
M-(I-\nabla \di)^{-1}(M-\nabla\upsilon)
\end{array}\right).\ee
\begin{remark}[Consistency of \eqref{e:blo}] The operator $(I-\nabla \di)^{-1}$ can be viewed as the Riesz isomorphism between the Hilbert spaces $E^*$ and $E$, where $E:=\{M\in X^{d\times d}| \di M\in X^d\}$ is equipped with the scalar product $(M,N)_{E}=(M,N)+(\di M,\di N)$, cf. \cite{Te79}. Consequently, \eqref{e:blo} defines a bounded linear operator on $X^{d(1+d)}$. \end{remark}
Set $$n=d(1+d),\ v=(u,\nabla u): [0,T]\to X^n\simeq X^d\times X^{d\times d},$$  $$L: D(L)\subset X_s^{n\times n}\to X^n, \quad L\,\left(\begin{array}{@{}c|c@{}}
M& \Upsilon \\  \hline 
\Upsilon^\top & \Xi
\end{array}\right)=\left(\begin{array}{@{}c@{}}
0 \\ \hline
-\di (\Upsilon^\top)+M+\widehat\Xi-\widecheck\Xi+\frac 1 2 I\tr M+\frac 1 2 I\tr \widehat\Xi
\end{array}\right).$$  

Let us now interpret the EPDiff equations as an abstract Euler equation. Denote $G:=\nabla u$, $G_0:=\nabla u_0$ in \eqref{epd1}--\eqref{epd3}. A tedious calculation shows that \eqref{epd1}--\eqref{epd3} is equivalent to \begin{gather} \notag \p_t u=  \di\Big[\p_t G+\di (G\otimes u)-u\otimes u-\widehat{(G\otimes G)}+\widecheck{(G\otimes G)}\\-\frac 1 2 I\tr u\otimes u-\frac 1 2 I\tr \widehat{(G\otimes G)}\Big],\label{epd4}\\
\label{epd5} G=\nabla u,\\
\label{epd6} u(0)=u_0,\, G(0)=G_0.\end{gather}
Tautologically, 
\begin{multline} \label{epd7} \p_t G= -\di (G\otimes u)+u\otimes u+\widehat{(G\otimes G)}-\widecheck{(G\otimes G)}+\frac 1 2 I\tr u\otimes u\\+\frac 1 2 I\tr \widehat{(G\otimes G)}  +\Big[\p_t G+\di (G\otimes u)-u\otimes u-\widehat{(G\otimes G)}+\widecheck{(G\otimes G)}\\-\frac 1 2 I\tr u\otimes u-\frac 1 2 I\tr \widehat{(G\otimes G)}\Big].\end{multline}
The system \eqref{epd4}--\eqref{epd7} can be rewritten as
\be  \label{e:aeulerepd}\p_t v= L(v\otimes v)+\left(\begin{array}{@{}c@{}}
 \di \xi \\ \hline
\xi
\end{array}\right),\quad v(0)=v_0,\ee where $$v(t)=\left(\begin{array}{@{}c@{}}
 u(t) \\ \hline
G(t)
\end{array}\right)\in P(X^n)$$ and \be \label{e:xie} \xi=\p_t G+\di (G\otimes u)-u\otimes u-\widehat{(G\otimes G)}+\widecheck{(G\otimes G)}\\-\frac 1 2 I\tr u\otimes u-\frac 1 2 I\tr \widehat{(G\otimes G)}.\ee Applying the projector $P$ to both sides of \eqref{e:aeulerepd}, we get \eqref{e:aeuler}. Reciprocally, \eqref{e:aeuler} implies \eqref{e:aeulerepd} where $\xi$ necessarily satisfies \eqref{e:xie} due to \eqref{epd7}.

A direct calculation shows that \eqref{e:acons} holds for $v=(u,\nabla u)$ sufficiently smooth, but the requirement \eqref{e:invt} is not met. 

The adjoint operator is
$$L^*: D(L^*)\subset X^n\to X_s^{n\times n}, \quad L^*\,\left(\begin{array}{@{}c@{}}
 \phi\\ \hline
\Phi
\end{array}\right)=\frac 1 2 \left(\begin{array}{@{}c|c@{}}
\Phi+\Phi^\top+I\tr \Phi & \nabla\Phi \\  \hline 
(\nabla \Phi)^\top & \widehat \Phi^*+\widehat{(\Phi^\top)}^*-\widecheck \Phi^*-\widecheck{(\Phi^\top)}^*+(\tr \Phi) \widecheck{I}^*
\end{array}\right).$$  If $(\phi,\Phi) \in P(X^n)$, then $\Phi=\nabla \phi$. If the eigenvalues of \be \label{e:maed} \left(\begin{array}{@{}c|c@{}}
\Phi+\Phi^\top+I\tr \Phi & \nabla\Phi \\  \hline 
(\nabla \Phi)^\top & \widehat \Phi^*+\widehat{(\Phi^\top)}^*-\widecheck \Phi^*-\widecheck{(\Phi^\top)}^*+(\tr \Phi) \widecheck{I}^*
\end{array}\right)(x),\quad \Phi=\nabla \phi,\ee are bounded from below, there is $k\ge 0$ such that $$\left(\begin{array}{@{}c|c@{}}
\Phi+\Phi^\top+(k+\tr \Phi)I & \nabla\Phi \\  \hline 
(\nabla \Phi)^\top & \widehat \Phi^*+\widehat{(\Phi^\top)}^*-\widecheck \Phi^*-\widecheck{(\Phi^\top)}^*+(k+\tr \Phi) \widecheck{I}^*
\end{array}\right)(x)\geq 0.$$ 
Taking the trace of the last block, we deduce that $k+\tr \Phi\geq 0$. Moreover,
the non-negativity of the principal minors of order $2$ yields $$(k+\tr\Phi+2\Phi_{ii})(2\Phi_{jj}-2\Phi_{ll}+k+\tr\Phi) \geq (\p_{x_i} \Phi_{jl})^2.$$ Letting $j=l$ and performing the summation w.r.t. to the remaining indices, we arrive at $$3(k+\tr\Phi)^2\ge (k+3\tr\Phi)(k+\tr\Phi) \geq |\nabla \tr \Phi|^2.$$ But $\Phi=\nabla \phi$, so $\int_{\mathbb T^d} \tr\Phi(y)\,dy=0$. As in the Camassa-Holm case above, Lemma \ref{l:gr} implies a uniform bound on $\tr \Phi$ and thus on the trace of the matrix in \eqref{e:maed}. This yields the trace condition, and leads to \begin{corollary} For every $u_0\in X^d$, there exists a generalized solution \eqref{e:geuler} to \eqref{epd1}--\eqref{epd3}. \end{corollary}
The restriction \eqref{e:goodv} holds provided $u_0$ has zero mean. 
\subsection{Euler-$\alpha$}
The Euler-$\alpha$ equations \cite{kapr,hmr2,mrs,sk} (with $\alpha=1$ for definiteness) may be written as
\begin{gather} \label{eal1} \p_t m+(\nabla u)^{\top}.m+\di (m\otimes u)+\nabla p=0,\\ \label{eal2}
m=u-\Delta u,\\
\di u=0, \label{eal3}\\ \label{eal4}
u(0)=u_0.\end{gather} The unknowns are $u:[0,T]\times \Omega\to \R^d$ and $p:[0,T]\times \Omega\to \R$. These equations are the geodesic equations (more precisely, the exponential map equations) on the group of volume-preserving diffeomorphisms with $H^1$ metric, see, e.g., \cite{KLMP}.  This example is quite similar to the previous one. We first recast \eqref{eal1}--\eqref{eal4} as \begin{gather} \notag \p_t u+\nabla p=  \di\Big[\p_t G+\di (G\otimes u)-u\otimes u-\widehat{(G\otimes G)}+\widecheck{(G\otimes G)}\\-\frac 1 2 I\tr u\otimes u-\frac 1 2 I\tr \widehat{(G\otimes G)}\Big],\label{eal5}\\
 \p_t G= -\di (G\otimes u)+u\otimes u+\widehat{(G\otimes G)}-\widecheck{(G\otimes G)}+\frac 1 2 I\tr u\otimes u \notag \\+\frac 1 2 I\tr \widehat{(G\otimes G)}  +\Big[\p_t G+\di (G\otimes u)-u\otimes u-\widehat{(G\otimes G)}+\widecheck{(G\otimes G)}\notag \\-\frac 1 2 I\tr u\otimes u-\frac 1 2 I\tr \widehat{(G\otimes G)}\Big].\label{eal6}\\
\label{eal7} G=\nabla u,\\
\label{eal8} \tr G=0,\\
\label{eal9} u(0)=u_0,\, G(0)=G_0,\end{gather}
cf. \eqref{epd4}--\eqref{epd7}. Set $$n=d(1+d),\ v=(u,G): [0,T]\to X^n\simeq X^d\times X^{d\times d},$$  $$L: D(L)\subset X_s^{n\times n}\to X^n, \quad L\,\left(\begin{array}{@{}c|c@{}}
M& \Upsilon \\  \hline 
\Upsilon^\top & \Xi
\end{array}\right)=\left(\begin{array}{@{}c@{}}
0 \\ \hline
-\di (\Upsilon^\top)+M+\widehat\Xi-\widecheck\Xi+\frac 1 2 I\tr M+\frac 1 2 I\tr \widehat\Xi
\end{array}\right).$$ 

Consider the set $$Y:=\mathcal P_\nabla X^n\cap \{(u,G)| \tr G=0\},$$ where $\mathcal P_\nabla$ was defined in \eqref{e:blo}. It is clear that $Y$ is a closed linear subspace of $X^n$. Let $$P:X^n\to Y$$ be the corresponding orthogonal projector. The orthogonal complement of $Y$ consists of the elements of the form $(\di \xi,
\xi+qI)$, $\xi\in X^{d\times d}$, $q\in X$. Rewrite the system \eqref{eal5}--\eqref{eal9} as
\be  \label{e:aeulereal}\p_t v= L(v\otimes v)+\left(\begin{array}{@{}c@{}}
 \di \xi \\ \hline
\xi+pI
\end{array}\right),\quad v(0)=v_0,\ee where $$v(t)=\left(\begin{array}{@{}c@{}}
 u(t) \\ \hline
G(t)
\end{array}\right)\in P(X^n)$$ and \be \label{e:ieal} \xi=\p_t G+\di (G\otimes u)-u\otimes u-\widehat{(G\otimes G)}+\widecheck{(G\otimes G)}\\-\frac 1 2 I\tr u\otimes u-\frac 1 2 I\tr \widehat{(G\otimes G)}-pI.\ee
Applying the projector $P$ to both sides of \eqref{e:aeulereal}, we get the abstract Euler equation \eqref{e:aeuler}. Reciprocally, \eqref{e:aeuler} implies \eqref{e:aeulereal} where $\xi$ necessarily satisfies \eqref{e:ieal} due to the trivial equality \eqref{eal6}.

As in the previous example, \eqref{e:acons} holds for $v=(u,\nabla u)$ sufficiently smooth. In contrast to EPDiff, \eqref{e:invt} is now valid since $$PL\left(\begin{array}{@{}c|c@{}}
qI& 0 \\  \hline 
0 & q \widecheck{I}^*
\end{array}\right)=P\left(\begin{array}{@{}c@{}}
0 \\ \hline
(d+1)qI
\end{array}\right)=0.$$ Thus we have \begin{corollary} For every $u_0\in X^d$, $\di u_0=0$, there exists a generalized solution \eqref{e:geuler} to \eqref{eal1}--\eqref{eal4}. \end{corollary}
Note that \eqref{e:goodv} holds provided $u_0$ has zero mean. 

\subsection{Incompressible isotropic Hookean
elastodynamics} The ``neo-Hookean'' model of motion of incompressible isotropic elastic fluid \cite{LLZ05,ST05,ST07,LSZ15,lei2015,lei2016} reads 
\begin{gather} \label{el1} \p_t u+\di (u\otimes u)+\nabla p=\di (F F^\top),\\ \label{el2} 
\p_t F+\di (F\otimes u)=(\nabla u) F,\\
\label{el3}  \di u=0,\\
\label{el4}  \di F^\top=0,\\ \label{el5} 
u(0)=u_0,\quad F(0)=F_0.\end{gather} The unknowns are $u:[0,T]\times \Omega\to \R^d$, $F:[0,T]\times \Omega\to \R^{d\times d}$ and $p:[0,T]\times \Omega\to \R$. Consider the projector  \be \label{e:bloel} P:X^{d\times d}\to X^{d\times d}, \quad \quad P(M)=\mathcal P_{d}(M):=
\left(\begin{array}{@{}c|c|c|c@{}}
 \mathcal P M_1 & \mathcal P  M_2 & \cdots & \mathcal P  M_d
\end{array}\right),\ee where $M_1,\dots,M_d\in  X^d$ are the columns of the matrix $M$. Obviously, $$\di (\mathcal P_d M)^\top=0,\quad M\in X^{d\times d}.$$
It is straigtforward to check that $$\di(\di (F\otimes u))^\top=\di((\nabla u) F)^\top,$$ which allows us to project \eqref{el2} onto $\mathcal P_{d}  X^{n\times n}$.  Set $$n=d(1+d),\ v=(u,F): [0,T]\to X^n\simeq X^d\times X^{d\times d},$$ $$P: X^n\to X^n,\quad P(\upsilon,\Phi)=(\mathcal P \upsilon, \mathcal P_{d} \Phi),$$ $$L: D(L)\subset X_s^{n\times n}\to X^n, \quad L\,\left(\begin{array}{@{}c|c@{}}
M& \Upsilon \\  \hline 
\Upsilon^\top & \Xi
\end{array}\right)=\left(\begin{array}{@{}c@{}}
\di \widehat{\Xi}-\di M \\ \hline
\di \widetilde\Upsilon-\di (\Upsilon^\top)
\end{array}\right).$$ Then \eqref{el1}--\eqref{el5} can be recast as the abstract Euler equation \eqref{e:aeuler}. The structure of equations \eqref{el1}--\eqref{el5} (in particular, their similarity with the ideal MHD equations) allows us to conjecture that they determine the geodesics on some Lie group. The conservativity condition \eqref{e:acons} holds for $v=(u,F)$ sufficiently smooth (this can be verified straightforwardly).   Let us check \eqref{e:invt}. Let $q\in X$ be a sufficiently smooth function. Then $$PL\left(\begin{array}{@{}c|c@{}}
qI& 0 \\  \hline 
0& q\widecheck I ^*
\end{array}\right)=P\left(\begin{array}{@{}c@{}}
d\nabla q-\nabla q \\  \hline 
0
\end{array}\right)=0.$$ As a result, we have \begin{corollary} For any $(u_0,F_0)\in X^d\times X^{d\times d}$ with $\di u_0=0,$ $\di F_0^\top=0$, there exists a generalized solution \eqref{e:geuler} to \eqref{el1}--\eqref{el5}. \end{corollary}
 Note that \eqref{e:goodv} holds provided $u_0$ and $F_0$ have zero mean (as a vector field and a matrix field, resp.). 
\subsection{Damping-free Maxwell's fluid} The motion of the  incompressible Old\-royd-B  viscoelastic material (also known as Jeffreys' fluid) is described \cite{O50,La13,ZV08} by the problem
\begin{gather} \label{maxw1} \p_t u+\di (u\otimes u)-\mu \Delta u+\nabla p=\di \tau,\\  \label{maxw2}
\p_t \tau+\di (\tau\otimes u)+Q(\nabla u,\tau)+a\tau=\frac 1 2 (\nabla u+(\nabla u)^\top),\\  \label{maxw3}
\di u=0,\\  \label{maxw4}
u(0)=u_0,\quad \tau(0)=\tau_0.\end{gather} The unknowns are $u:[0,T]\times \Omega\to \R^d$, $\tau:[0,T]\times \Omega\to \R^{d\times d}_s$  and $p:[0,T]\times \Omega\to \R$. When the retardation time vanishes, we get Maxwell's fluid (this corresponds to $\mu=0$). 
The choice $a=0$ (cf. \cite{LLZ05,Z18}) tallies with the damping-free case when the relaxation time blows up. We restrict ourselves to the purely hyperbolic case $a=\mu=0$, which coheres with a purely elastic fluid. Note that (cf. \cite{LLZ05,Lin16}) the purely hyperbolic system with $Q=-\nabla u\tau-\tau(\nabla u)^\top$(the upper-convective case) can be made equivalent  to  \eqref{el1}-\eqref{el5} if one assumes the ans\"atze \be \label{ans} \tau=F F^\top,\quad \di F^\top=0.\ee This makes sense because the constraints \eqref{ans} are preserved along the flow. Here we  assume neither $\eqref{ans}$ nor even positive-definiteness of $\tau$. The term $Q$ is related to frame-invariance and is known to create mathematical difficulties. We consider the simplified model \cite{ER15} with $Q\equiv 0$, see also \cite{LLZ05,Z18,ZV08,La13}. This model, unlike \eqref{el1}--\eqref{el5}, is not frame-indifferent, but it is invariant to the transformations which keep the frame inertial (e.g., to the Galilean transformation). We arrive at the following conservative problem: 
\begin{gather}  \label{maxw5} \p_t u+\di (u\otimes u)+\nabla p=\di \tau,\\ 
\label{maxw6}
\p_t \tau+\di (\tau\otimes u)=\frac 1 2 (\nabla u+(\nabla u)^\top),\\  \label{maxw7}
\di u=0,\\
 \label{maxw8}
u(0)=u_0,\quad \tau(0)=\tau_0.\end{gather} Set $$n=d+d(d+1)/2,\ v=(u,\tau): [0,T]\to X^n\simeq X^d\times X^{d\times d}_s,$$ $$P: X^n\to X^n,\quad P(\upsilon,\varsigma)=(\mathcal P \upsilon,\varsigma),$$ $$A: D(A)\subset X^n\to X^n, \quad A(\upsilon,\varsigma)=-\frac 1 2\left(2\di \varsigma,  \nabla \upsilon+(\nabla \upsilon)^\top\right),$$ $$L: D(L)\subset X_s^{n\times n}\to X^n, \quad L\,\left(\begin{array}{@{}c|c@{}}
M& \Upsilon \\  \hline 
\Upsilon^\top & \Xi
\end{array}\right)=\left(\begin{array}{@{}c@{}}
-\di M \\ \hline
-\di (\Upsilon^\top)
\end{array}\right).$$ Then \eqref{maxw5}--\eqref{maxw8} can be written in the abstract form \eqref{e:aeulerd}. Condition \eqref{e:aconsd} follows by integration by parts. Moreover, \eqref{e:invt} is satisfied since $$PL\left(\begin{array}{@{}c|c@{}}
qI& 0 \\  \hline 
0 & q \widecheck{I}^*
\end{array}\right)=P\left(\begin{array}{@{}c@{}}
-\di (qI) \\ \hline
0
\end{array}\right)=\left(\begin{array}{@{}c@{}}
-\mathcal P\nabla q \\ \hline
0
\end{array}\right)=0$$ for each $q\in X$ sufficiently smooth. In light of Remark \ref{r:aoper}, we have the following corollary:
\begin{corollary} For any $(u_0,\tau_0)\in X^d\times X^{d\times d}_s$ with $\di u_0=0$, there exists a generalized solution \eqref{e:geuler} of the extended system \eqref{e:aeulerdd} tantamount to \eqref{maxw5}--\eqref{maxw8}. \end{corollary}
\subsection{Korteweg-de Vries and Zakharov-Kuznetsov} \label{KdV}
Let $\Omega=\mathbb T^1$. The Korteweg-de Vries equation is \be\label{e:kdv} \p_t v+v_{xxx}=6 v v_x, \quad v(0)=v_0.\ee
The unknown is $v:[0,T]\times \Omega\to \R$. It is the Euler-Arnold equation for the Virasoro group \cite{KW08}. The Korteweg-de Vries equation is known to be globally well-posed \cite{CT03} but we still consider this example for the sake of curiosity.
Set $$n=1,\quad P=I, \quad A: D(A)\subset X\to X, \quad A(\upsilon)=\upsilon_{xxx},$$ $$L: D(L)\subset X\to X, \quad L(\sigma)=3\sigma_x.$$ Then \eqref{e:kdv} can be written in the abstract form \eqref{e:aeulerd}. Condition \eqref{e:aconsd} can be easily verified via integration by parts. However, \eqref{e:invt} is not satisfied. 

As in Remark \ref{r:aoper}, consider the extended problem \eqref{e:aeulerdd} with $$\tilde P(\upsilon,a)=\left(\upsilon,\int_\Omega a\,d\mu\right),$$
$$\tilde L\left(\begin{array}{@{}cc@{}}
\sigma &  z \\  
z & a 
\end{array}\right)=\left(\begin{array}{@{}c@{}}
3\sigma_x-z_{xxx}\\  
0 
\end{array}\right).$$ The adjoint operator is $$\tilde L^*: D(\tilde L^*)\subset X^2\to X_s^{2\times 2}, \quad \tilde L^*\left(\begin{array}{@{}c@{}}
\phi \\  
\psi
\end{array}\right)=\frac 1 2\left(\begin{array}{@{}cc@{}}
-6\phi_x& \phi_{xxx}\\  
\phi_{xxx} & 0 
\end{array}\right).$$ If  there is $k\geq 0$ such that $$\left(\begin{array}{@{}cc@{}}
-6\phi_x+k& \phi_{xxx}\\  
\phi_{xxx} & k 
\end{array}\right)\geq 0,$$ then $$-6k\phi_x+k^2-\phi_{xxx}^2\geq 0.$$ Consequently, $$ \int_{\mathbb T^1}\phi_{xxx}^2\leq k^2.$$ By elliptic regularity, $\phi_x$ is uniformly bounded in $W^{2,2}(\mathbb T^1)$ and thus in $L^\infty(\mathbb T^1)$. Accordingly, the trace of  $\tilde L^*(\phi,\psi)$  is uniformly bounded, which implies the trace condition. 

\begin{corollary} \label{c:kdv} For any $v_0\in X$, there exists a generalized solution \eqref{e:geuler} of the extended system \eqref{e:aeulerdd} tantamount to \eqref{e:kdv}. \end{corollary}

\begin{remark}[Zakharov-Kuznetsov equation] The considerations above, including Corollary \ref{c:kdv}, can be reproduced almost verbatim for the Cauchy problem for the Zakharov-Kuznetsov equation \cite{ZK, LL13, HK13} on $\Omega=\mathbb T^d$, $d=2,3$, that reads \be\label{e:zk} \p_t v+\partial_{x_1}\Delta v=v \partial_{x_1} v, \quad v(0)=v_0.\ee Unlike KdV, \eqref{e:zk} is not completely integrable. To the best of our knowledge, the global well-posedness with  $v_0\in X=L^2(\mathbb T^d)$ has yet not been proved,  cf. \cite{K21}. However, for $\Omega=\mathbb R^d$, $d=2,3$, the global well-posedness for the $L^2$ data has recently been established in \cite{K19,HK21}.  \end{remark}

\begin{exmp}[Cnoidal waves] The simplest strong solutions to KdV on $\mathbb{T}$ are the cnoidal waves $$u(t,x)=c_1 +c_2 \mathrm{cn} (c_3 x- c_4 t; m),$$ where $\mathrm{cn}$ is the elliptic cosine, with appropriate constants $c_1,c_2,c_3,c_4, m$. Note that \begin{equation}\label{e:maxims} s:=\max_{x\in \mathbb{T}} u_x(t,x)\end{equation} does not depend on $t$. Let  $$\tilde v(t)=\left(\begin{array}{@{}c@{}}
v(t) \\  
a(t)
\end{array}\right)\in \tilde P(X^2)$$ be the generalized solution provided by Corollary \ref{c:kdv}. Here $a$ does not depend on $x$ due to the definition of $\tilde P$. If $T$ is not too large, and $t\in (0,T)$, we have $$\frac 1 {T-t} I\geq \frac 1 T I\geq \left(\begin{array}{@{}cc@{}}
6 u_x& - u_{xxx}\\  
- u_{xxx} & 0 
\end{array}\right)=-2 \tilde L^* \left(\begin{array}{@{}c@{}}
u \\  
1
\end{array}\right),$$ whence \eqref{e:pd++} holds and Theorem \ref{t:smooth} is applicable. However, for $T>\frac 1 {6s}$ and very small $t>0$, we have $1-6(T-t)s< 0$ and $1-6(T-t)u_x< 0$ in a neighborhood of the maximizer in \eqref{e:maxims}. Consequently, $I\not\geq 2(t-T)\tilde L^*(u(t),1)$ on a set of positive measure in $(0,T)\times \mathbb T$.  Thus, if $T>\frac 1 {6s}$, the pair $(v,a)$  does not coincide with $(u,1)$ by Corollary \ref{c:enlevels}. Moreover, $v$ and $a$ either violate at least one of the initial conditions $v(0)=v_0$, $a(0)=1$ or do not satisfy at least one of the identities $\p_t v+v_{xxx}=6 v v_x$, $\p_t a=0$ (but we do not know how many and which of the described scenarios are actually true). 
\end{exmp}
\subsection{The template matching equation} This example is the last but definitely not the least. It interpolates between the two prototype examples from Section \ref{eulhj}, and violates the trace condition. The template matching equation, also known as the inviscid multidimensional  right-invariant Burgers equation, reads \begin{gather} \label{tm1} \p_t u+\di (u\otimes u)+\frac 1 2 \nabla |u|^2=0,\\
\label{tm3} u(0)=u_0.\end{gather}
The unknown is $u:[0,T]\times \Omega\to \R^d$. It describes the geodesics of the diffeomorphism group with right-invariant $L^2$ metric, see, e.g., \cite{KLMP,ra18,hm01}. 

Set $$n=d,\ P=I,$$  $$L: D(L)\subset X_s^{n\times n}\to X^n, \quad L(M)=
 -\di M-\frac 1 2 \nabla \tr M.$$  It is easy to see that \eqref{e:acons} formally holds, 
and \eqref{tm1}--\eqref{tm3} is equivalent to the abstract Euler equation \eqref{e:aeuler} with $P$ and $L$ just defined. The adjoint operator is $$L^*: D(L^*)\subset X^n\to X_s^{n\times n}, \quad L^*(\psi)=\frac 1 2 (\nabla \psi+(\nabla \psi)^\top+(\di \psi) I);$$ consequently, the trace condition is not valid. However, $PL(I)=0$, thus by Theorem \ref{t:exweak} there is a solution $(E,B)$ to the dual problem. The results of Section \ref{s:uniq} are applicable, but Theorem \ref{t:ldistr} and Corollary \ref{c:enlevels} are only relevant if the maximizer $(E,B)$ belongs to the restricted regularity class \eqref{e:be}. Note also that \eqref{e:goodv} holds provided $u_0$ has zero mean.

\subsection*{Acknowledgment} The author wishes to thank the anonymous referees for many valuable comments, which led to important improvements of the manuscript.  The author is very grateful to Yann Brenier and Wenhui Shi for inspiring discussions on the subject. 
The research was partially supported by the Portuguese Government through FCT/MCTES and by the ERDF through PT2020 (projects UID/MAT/00324/2020,\\ PTDC/MAT-PUR/28686/2017 and TUBITAK/0005/2014). 

\subsection*{Conflict of interest statement} We have no conflict of interest to declare.

\end{document}